\documentclass[12pt]{amsart}
\usepackage{epsfig,color,wrapfig}
\usepackage{graphicx}
\usepackage{float}
\usepackage{subcaption}

\makeatother

\theoremstyle{definition}

\newtheorem*{thm*}{Theorem}
\newtheorem{thm}{Theorem}[section]
\newtheorem{cor}[thm]{Corollary}

\newtheorem{lem}[thm]{Lemma}

\newtheorem{defn}[thm]{Definition}

\newtheorem{rmk}[thm]{Remark}
\newtheorem{prop}[thm]{Proposition}

\numberwithin{equation}{section}

\title{Stability of regular shrinkers in the network flow}
\author{Jui-En Chang}
\thanks{The author is supported by MOST-107-2115-M-002-015-MY3}
\address{National Taiwan University, Department of Mathematics\\
No.1, Sec. 4, Roosevelt Rd., Da'an Dist., Taipei City, Taiwan, 10617}
\email{jechang@ntu.edu.tw}

\date{\today}

\usepackage{latexsym, amsmath,amssymb}
\usepackage{amsmath}
\usepackage{amsfonts}
\usepackage{amssymb}
\usepackage{amscd}
\usepackage{amsthm}
\usepackage{amsbsy}
\usepackage{graphicx}
\usepackage{bm}
\usepackage{epsfig,epstopdf,url,array}

\begin{document}
\maketitle
\begin{abstract} 
The singularities of network flow are modeled by self-similarly shrinking solutions called regular shrinkers. In this paper, we study the stability of regular shrinkers. We show that all regular shrinkers with two or more enclosed regions can be perturbed away. Among the regular shrinkers with one enclosed region, 4-ray star, 5-ray star, fish, and rocket are unstable.
\end{abstract}

\section{Introduction}
The network flow is a geometric flow that studies the flow of a network, an essentially singular set, in $\mathbb{R}^2$. This flow is first proposed by Mullins \cite{Mu}.
It has several applications. In material science, it models the behavior of grain boundary of a multicrystalline material. It is also the first attempt to study a flow on an essentially singular geometric object. The network flow has several different behaviors which are not shown in the smooth counterpart, the curve shortening flow.

To make the flow problem well-posed, we impose the Herring condition: All multi-junctions are triple-junctions with angles between the curves being $\frac{2\pi}{3}$. For the most simple case that a network with only one triple junction, Bronsard and Reitich in \cite{BR} establish the short time existence and uniqueness. After their contribution, more complicated cases are considered. Mantegazza, Novaga, and Pluda in \cite{MNP2} establish existence and uniqueness for general networks. About more study of the network flow, the reader can refer to \cite{INS,MMN,MNP,MNP2,MNPS,MNT}, especially, \cite{MNP2} and \cite{MNPS}. 

At the maximal time of existence, the singularity may occur. Using parabolic scaling, the tangent flow at the singularity is a self-similarly shrinking solution. If we translate the space and time variables such that the singularity happens at the origin when time $t=0$. The time $t=-1$ slice is a regular network that satisfies
\begin{equation}
    k+\frac{\langle x,N\rangle}{2}=0,
\end{equation}
where $N$ is unit normal, $x$ is the position and $k$ is the curvature with respect to $N$. 

By finding regular networks satisfying the equation, we can limit the possibility of singularities. We call such a network a regular shrinker. They describe the possible shape of the singularities. If there are no triple junctions, it reduces to the case in the curve shortening flow. Abresch and Langer \cite{AL} classify all immersed solutions and show that the only embedded self-similarly shrinking curves are a line through the origin or a circle centered at the origin. In the presence of triple junctions, there are two solutions with exactly one triple junction. One of them is the standard triod. The other solution is the Brakke spoon which is first described in the work of Brakke \cite{B}. The classification of regular shrinker with one enclosed region is done by Chen and Guo \cite{CG}. Baldi, Haus, and Mantegazza \cite{BHM, BHM2} exclude the $\Theta$-shaped network. Following their work, the author and Lue \cite{CL} show that there is only one regular shrinker with exactly two enclosed regions. The appendix of \cite{MNPS} contains a collection of all known regular shrinkers and some possible numerical results. For more complicated cases, even though it is conjectured there are only finitely many regular shrinkers, a complete classification is still hard to obtain.

The regular shrinkers with no more than one enclosed region play important roles in this article. Their pictures and names are shown below.
\begin{figure}[H]
    \centering
    \includegraphics[width=3cm]{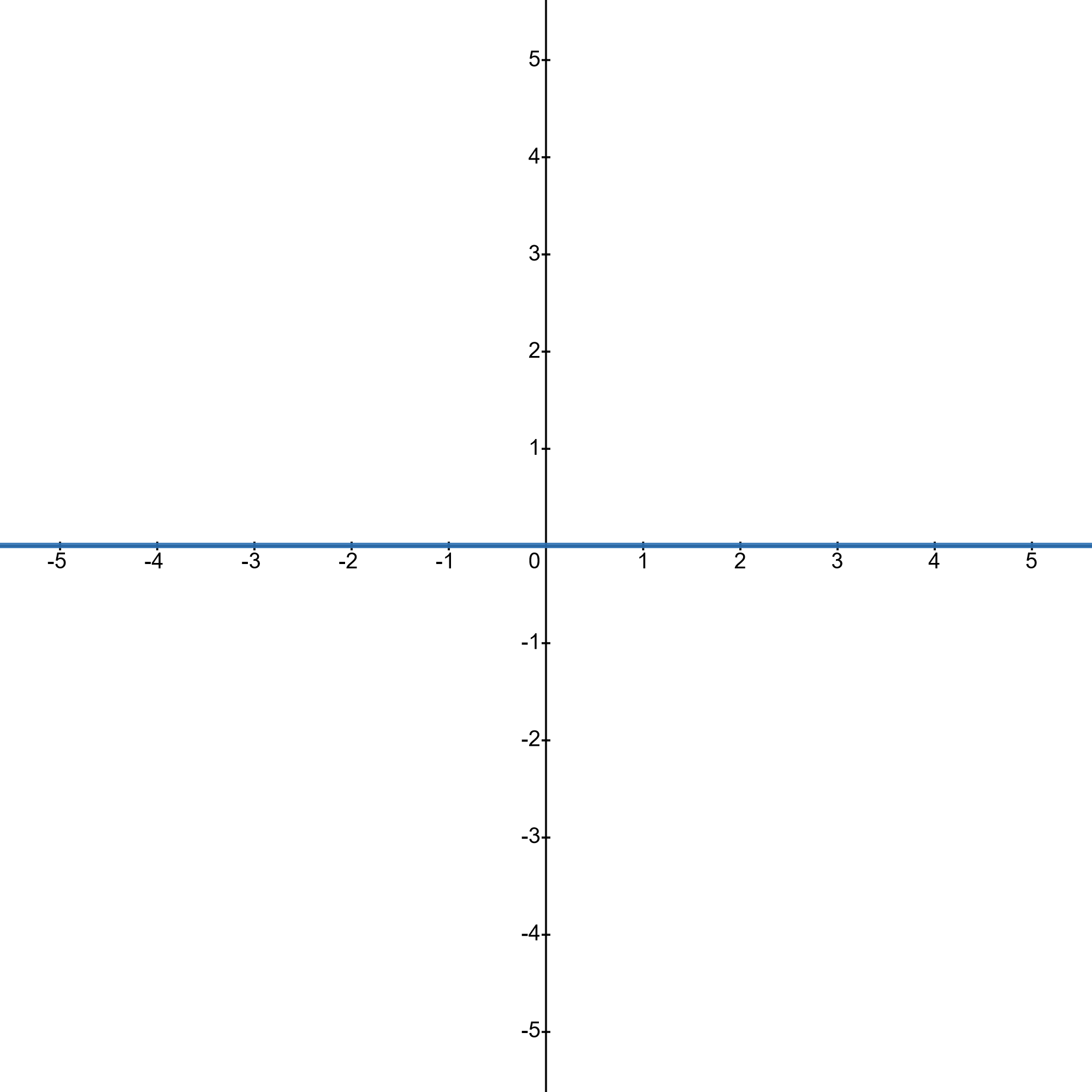}
    \includegraphics[width=3cm]{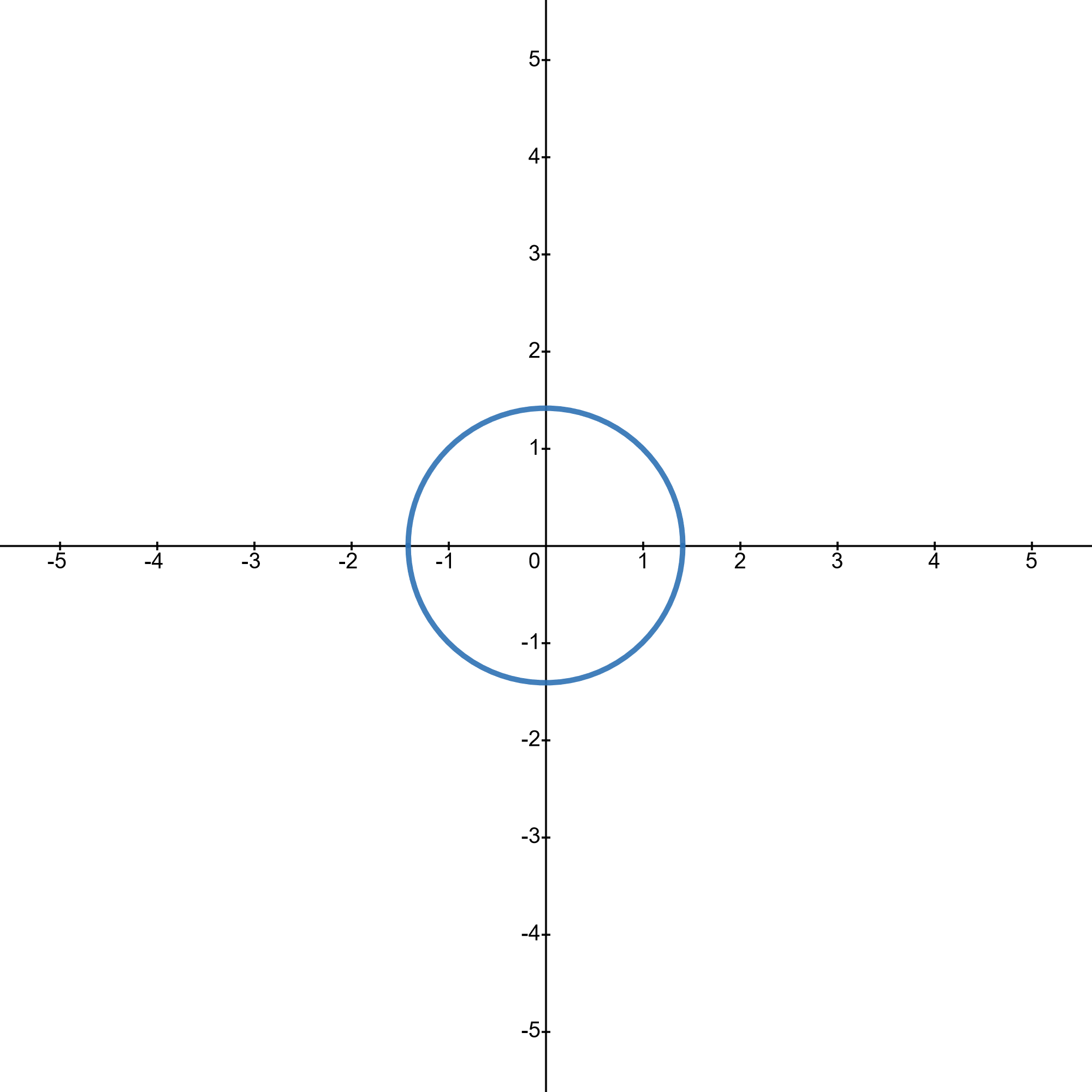}
    \includegraphics[width=3cm]{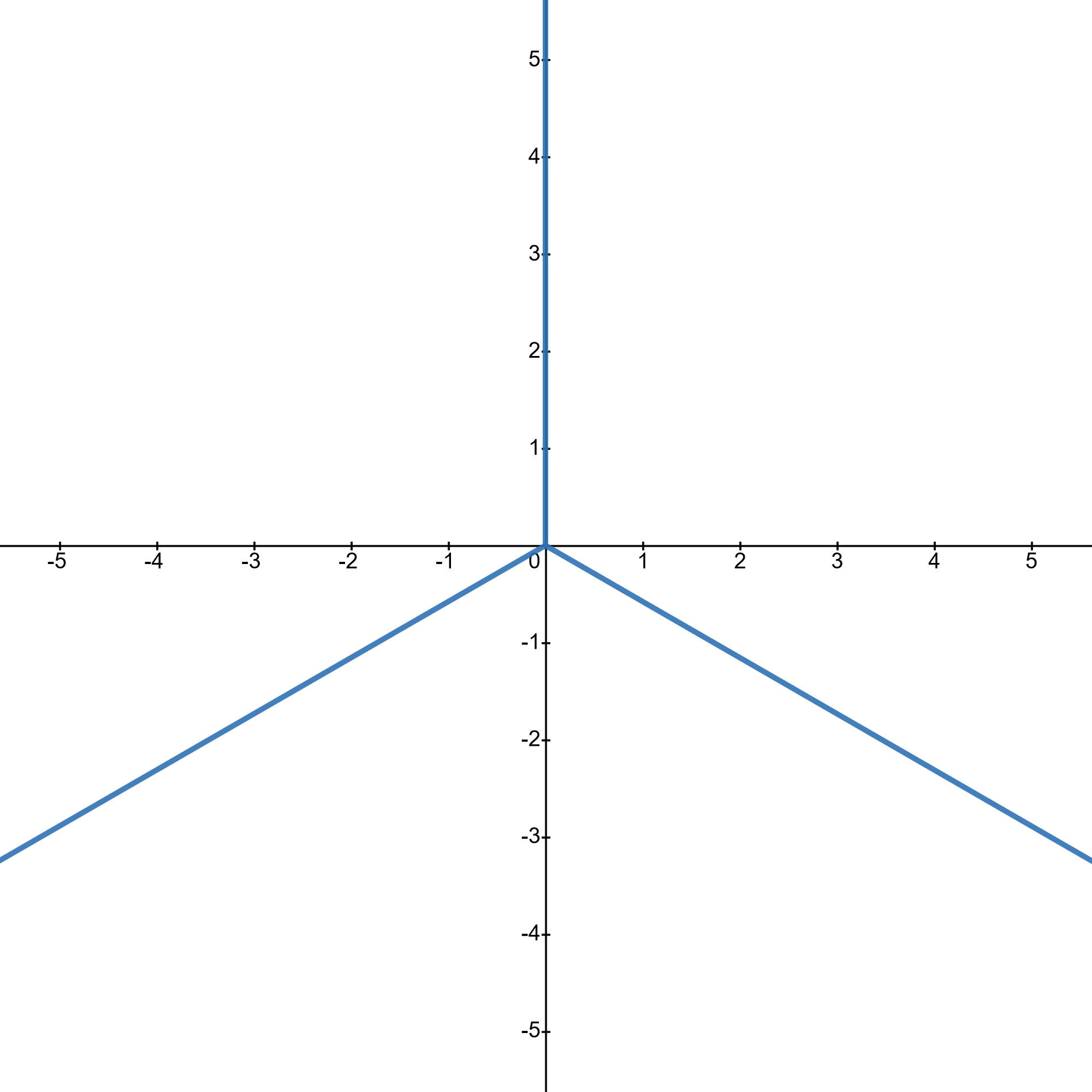}
    \includegraphics[width=3cm]{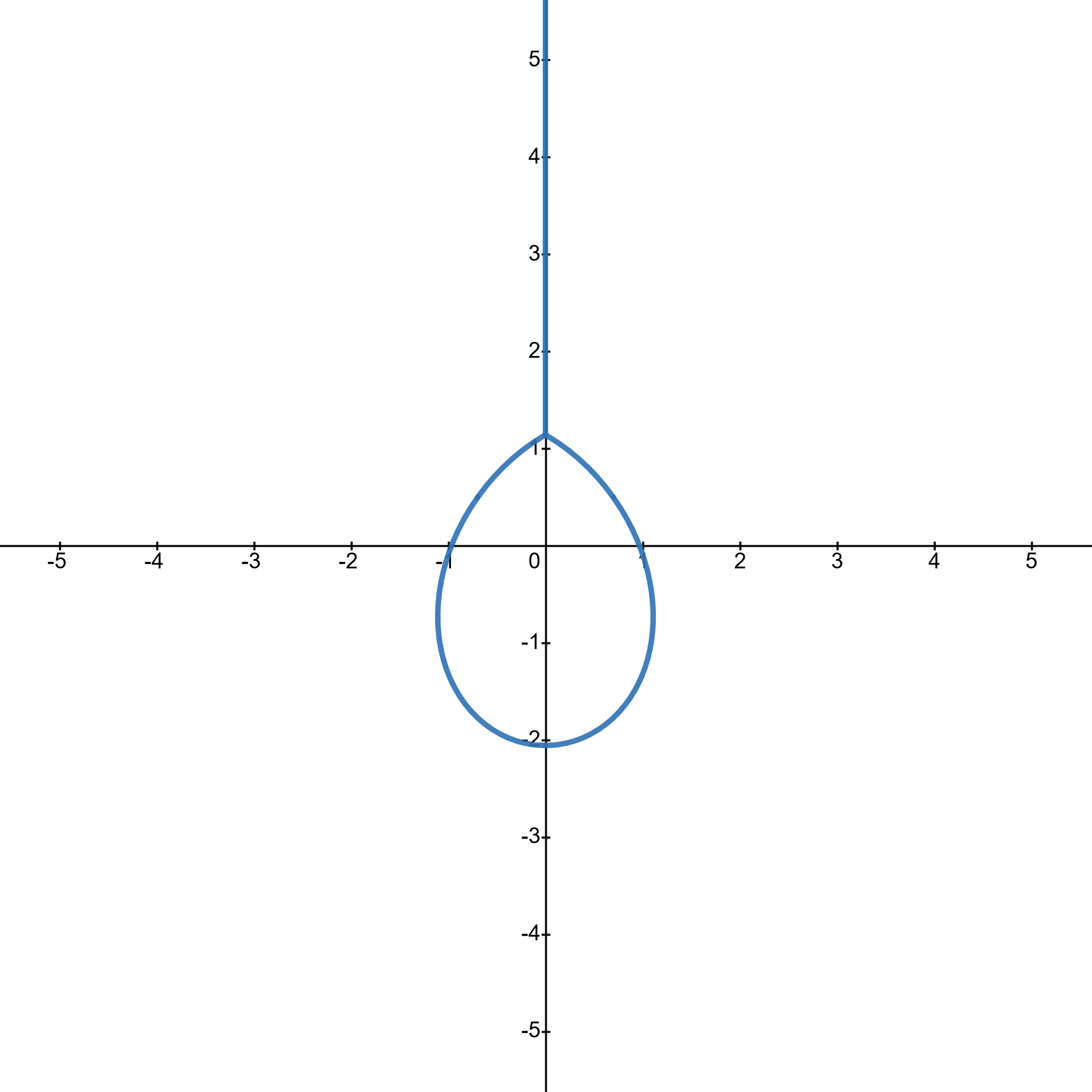}    \includegraphics[width=3cm]{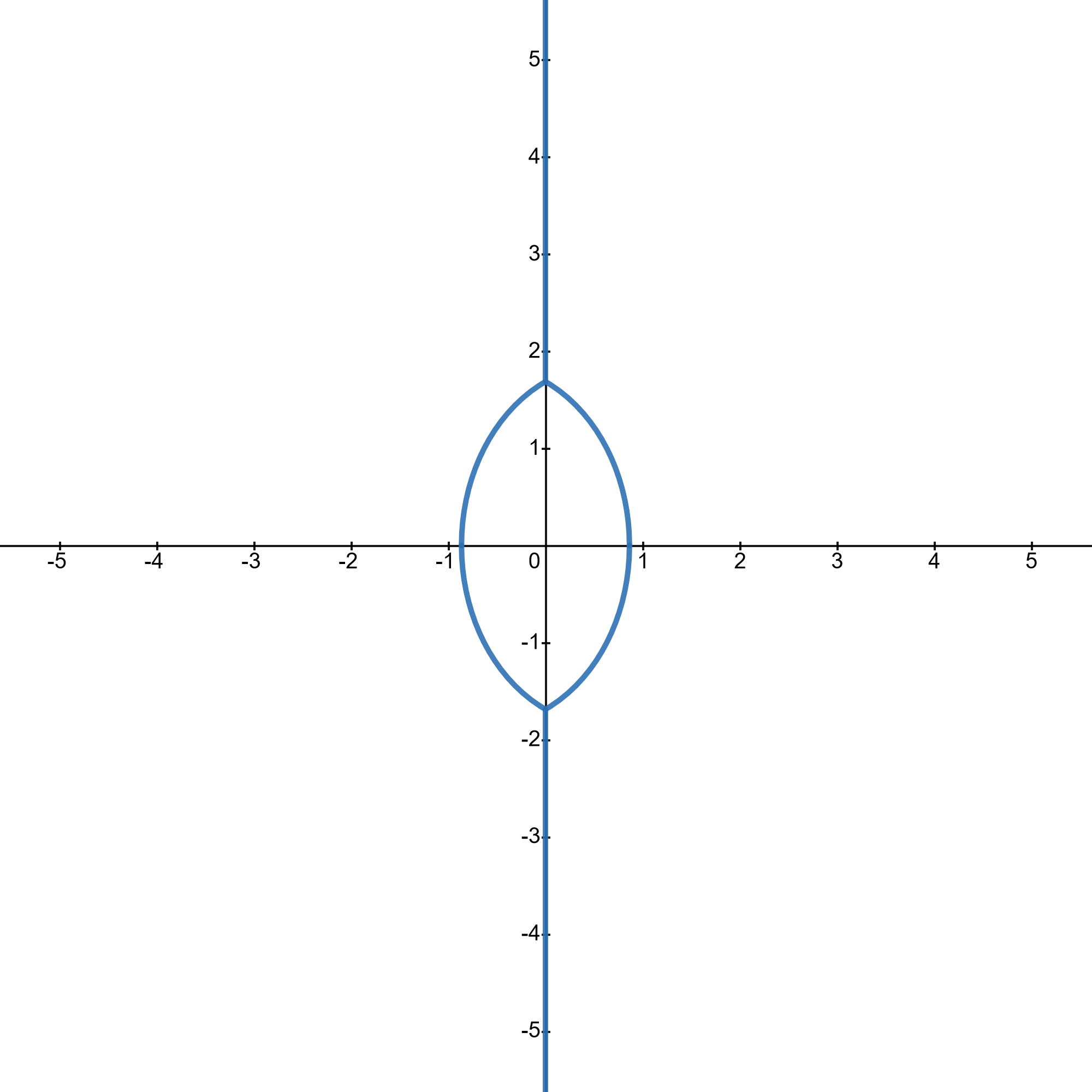}
    \caption{Line, circle, standard triod, Brakke spoon, lens}
\end{figure}
\begin{figure}[H]
    \centering
    \includegraphics[width=3cm]{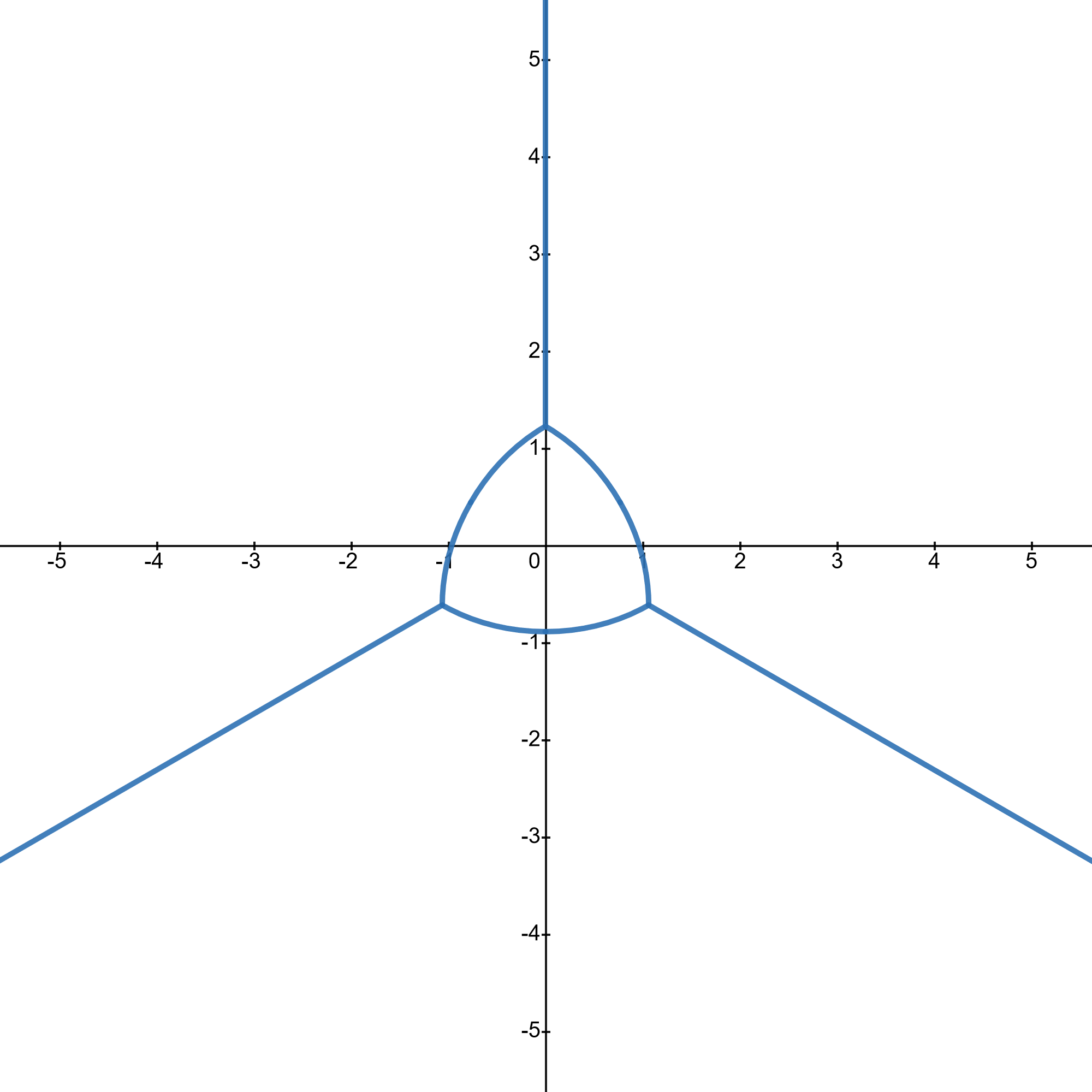}   
    \includegraphics[width=3cm]{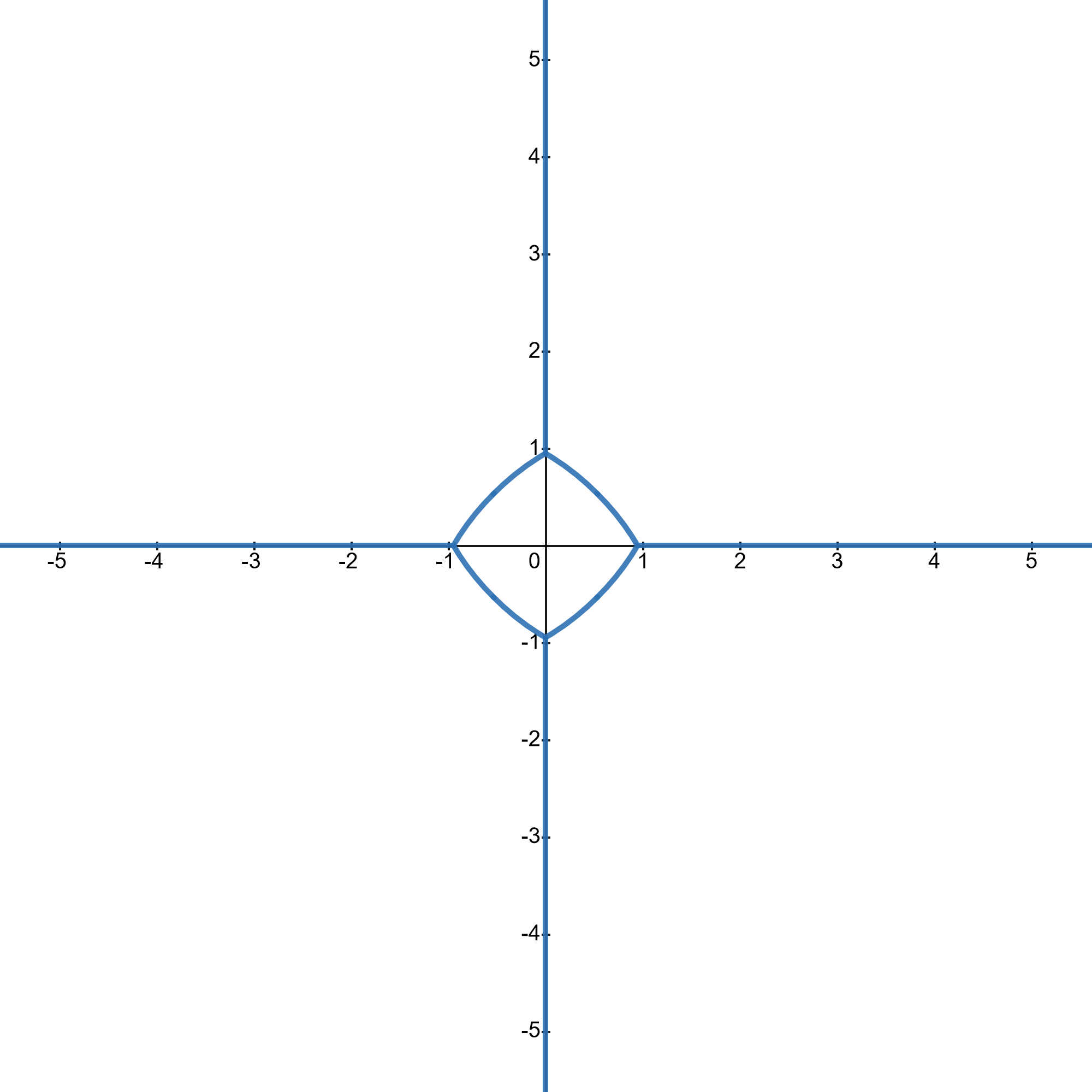}
    \includegraphics[width=3cm]{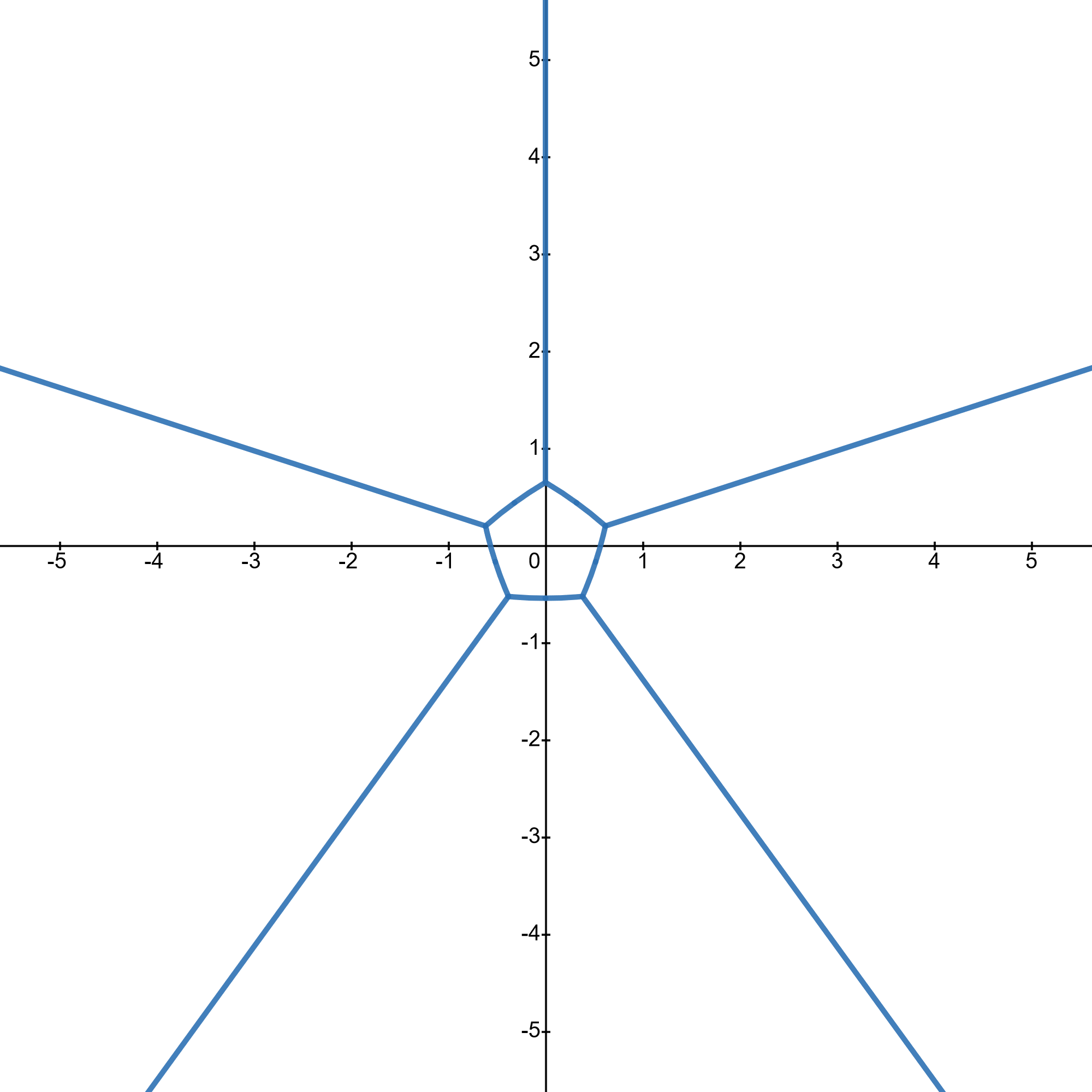}
    \includegraphics[width=3cm]{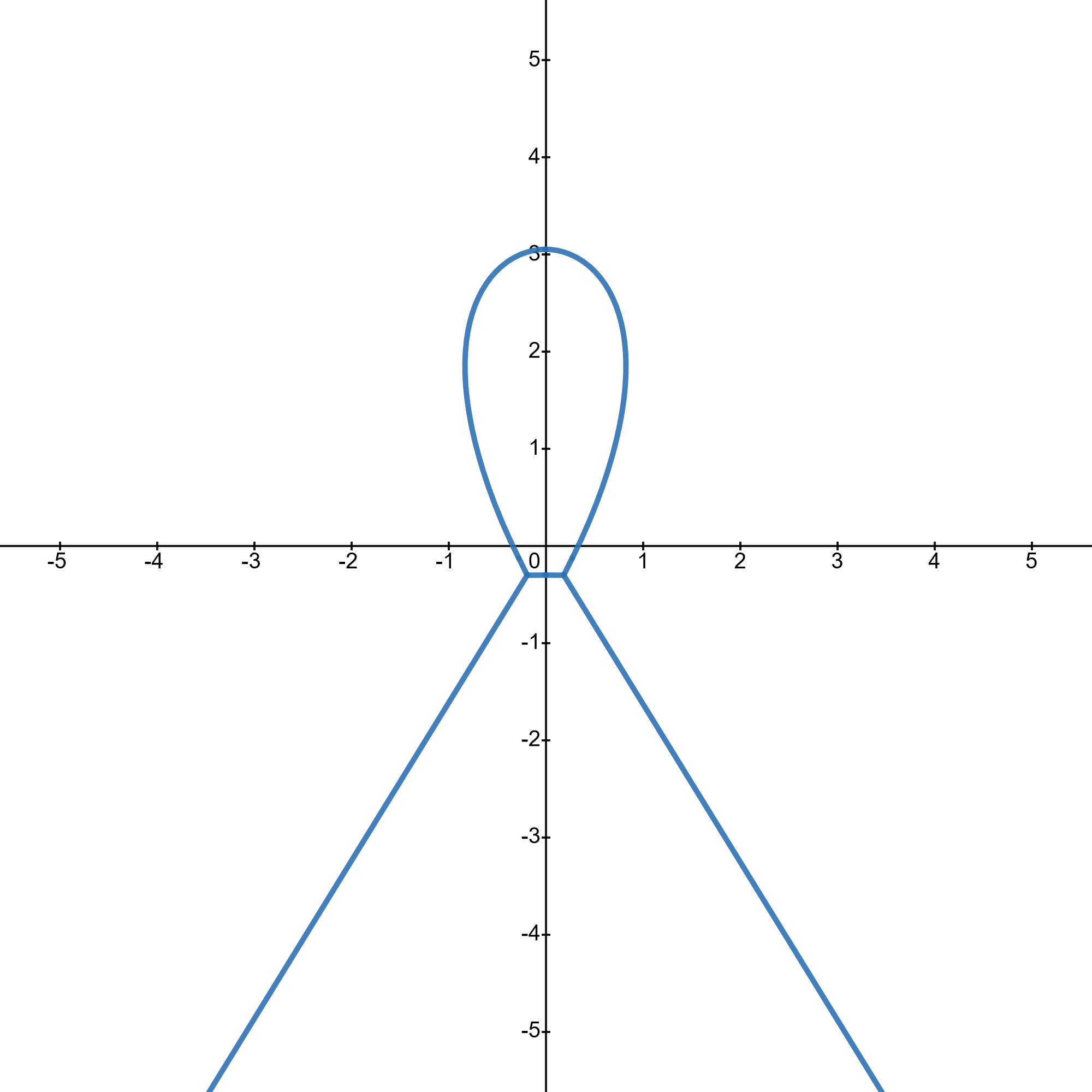}
    \includegraphics[width=3cm]{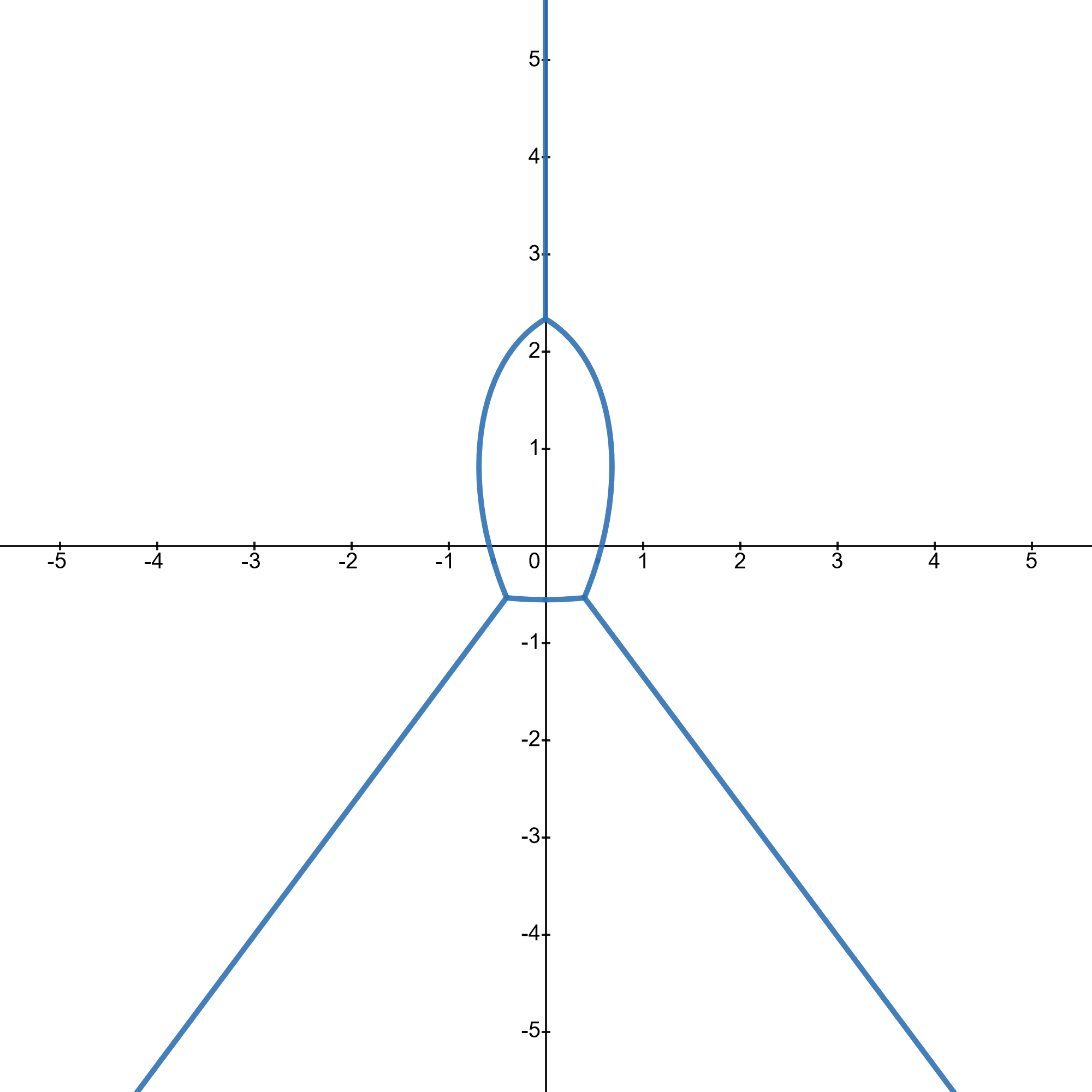}
    \caption{3-ray star, 4-ray star, 5-ray star, fish, rocket}
\end{figure}

Even though there are lots of regular shrinkers, not all of them are likely to appear as the tangent flow of a singularity. Near a singularity, if we perturb the network, will the singularity of the new network has the same tangent flow as the original network? In the study of network flow, there are some affirmative results that some regular shrinkers are stable. 

In \cite{EW}, Epstein and Weinstein use a perturbation in the space of curvature to show that only the circle with multiplicity 1 is the only stable compact self-shrinker in $\mathbb{R}^2$. The Brakke spoon is shown to be the blow-up limit for all spoon-shaped networks in the work of Pluda \cite{P}. This implies stability since any perturbation of the Brakke spoon is topologically spoon-shaped. The lens is shown to be the rescaling limit of any flow starting from a symmetric lens-shaped network in \cite{AGH} and the work of Bellettini and Novaga \cite{BN}. Since they require the network to be symmetric, the problem of general networks which are topologically equivalent to the lens is still open.

In this paper, we will show that unlike the above regular shrinkers, some regular shrinkers are unstable. We can perturb them in a way such that they can not be the tangent flow of the perturbed network. To establish this, let $x_0\in\mathbb{R}^2$ and $t_0>0$, the functional $F_{x_0,t_0}$ is defined to be
\begin{equation}
    F_{x_0,t_0}(\Gamma)=\frac{1}{\sqrt{4\pi t_0}}\int_\Gamma e^{-\frac{|x-x_0|^2}{4t_0}}d\sigma.
\end{equation}
The $F$-functional is important in the study of self-similarly shrinking solutions in mean curvature flow. It also works for network flow. A network $\Gamma$ is a critical point of $F_{x_0,t_0}$ if and only if it is the $t=-t_0$ slice of a self-similarly shrinking network with respect to the point $x=x_0$. We define the entropy $\lambda$ of a network $\Gamma$.
\begin{equation}
    \lambda(\Gamma)=\sup_{x_0,t_0}F_{x_0,t_0}(\Gamma).
\end{equation}
The entropy has the property that the critical points of $\lambda$ are regular shrinkers for the network flow and if $\Gamma_t$ flows under network flow, $\lambda(\Gamma_t)$ is nonincreasing. Therefore, if we can perturb a regular shrinker so that the entropy $\lambda$ decreases, at the singular time, the network flow has even lower entropy and the original network cannot be the tangent flow of the singularity.

To proceed, we need the idea of $F$-stability and entropy stability. The $F$-stability of a regular shrinker is defined as follows
\begin{defn}
A regular shrinker $\Gamma$ for $F_{x_0,t_0}$ is $F$-stable if for every compactly supported variation $\Gamma_s$, there exists variation $x_s$ of $x_0$ and $t_s$ of $t_0$ that makes $\frac{d^2}{ds^2}(F_{x_s,t_s}(\Gamma_s))\geq0$ at $s=0$.
\end{defn}
Also, we say that a regular shrinker is entropy-stable if it is a local minimum for the entropy functional $\lambda$.

The main theorem of this paper is the $F$-unstableness of certain regular shrinkers.
\begin{thm}
The 4-ray star, 5-ray star, fish, and rocket are $F$-unstable regular shrinkers.
\end{thm}
From this theorem, we can use the result in \cite{CM} to establish that the regular shrinkers in the theorem are entropy-unstable. Therefore, we can perturb them such that they can not be the tangent flow of the perturbed network.

The regular shrinkers with two or more enclosed regions are even more unstable. Via a different approach, we also show that there is a way to perturb them such that they cannot be the tangent flow of the singularity of the new network.
\begin{thm}\label{thm:2region}
For a regular shrinker with two or more enclosed regions, there is a perturbation such that the tangent flow of the perturbed network is not the same as the original regular shrinker.
\end{thm}

\subsection{Structure of this paper}

The paper is organized as follows. In section \ref{sec:2region}, we obtain an immediate result that the regular with two or more enclosed regions can be perturbed away. We don't need to use $F$-functional at this stage since an argument considering the area of each enclosed region is sufficient. In section \ref{sec:F_functional}, we introduce the $F$-functional. Sections \ref{sec:1st_var} through \ref{sec:2nd_var} focus on computing the first and the second variation formula of the $F$-functional. We can obtain an eigenfunction problem from the second variation formula. Section \ref{sec:eigunfunction} describes the eigenfunction problem and the eigenfunctions corresponding to translation in space and in time. They are the eigenfunctions with important geometry meaning. After that, we can deal with individual regular shrinkers. Sections \ref{sec:stars} and \ref{sec:fish_rocket} deals the variation in the case of 4-ray star, 5-ray star and the case of fish, rocket, respectively. Finally, section \ref{sec:F_unstable} establishes the $F$-unstableness by cutting off the function of variation to obtain compact variations which satisfies the condition in the definition of $F$-stability. Section \ref{sec:appendix} collect some previous numerical results of known regular shrinkers needed in the proof.

\subsection{Notations}
Throughout this paper, we will use the following notations. $\Gamma$ will denote a possibly open network. Let $O_i$ be the multi-junctions of $\Gamma$. Note that in section \ref{sec:1st_var} we don't assume $\Gamma$ to be regular. In that section, the angle between curves may not be $\frac{2\pi}{3}$ and there may be more than 3 curves meeting at a multi-junction $O_i$.

On a curve $\gamma$, let $T$ be a unit tangent vector field and $N$ be a unit normal vector field. We don't require $\{T,N\}$ to be positively oriented. Let $k$ denotes the curvature with respect to $N$, i.e. $\nabla_TT=kN$. Note that the sign of $k$ depends on the choice of $N$ but it is independent of the choice of $T$.

At a multi-junction $O_i$, let $\gamma_i^j$ be the curves with endpoint $O_i$. We use $\hat{T}_i^j$ to denote the unit tangent vector at $O_i$ which points towards the $j$th curve. For a function $v$ which is continuous on each curve, we use $v_i^j=\lim_{P\in\gamma_i^j,P\to O_i}v(P)$ to denote the boundary value seen from the curve $\gamma_i^j$.

For any vector field $V$ on the network, we decompose it to the normal part and the tangent part $V=vN+v^TT$. We often use capital letters to denote a vector field and lowercase letters to denote a scalar function except for the position vector $x$ and its variation $y$, $z$.

\section{Unstableness of regular shrinkers with two or more enclosed regions}\label{sec:2region}
We start from the most unstable case: regular shrinkers with two or more enclosed regions. First, we need a lemma to describe the area decreasing rate of an enclosed region under network flow. This is a well-known lemma. We include the statement and the proof here for completeness.
\begin{lem}
If an enclosed region has $m$ edges, the decreasing rate of area $a$ is
\begin{equation}\nonumber
\frac{da}{dt}=(m-6)\frac{\pi}{3}.
\end{equation}
\end{lem}
\begin{proof}
\begin{equation}\nonumber
-\frac{da}{dt}=\int kds=2\pi-\sum{\phi_i}=2\pi-m\frac{\pi}{3}.
\end{equation}
By Gauss-Bonnet theorem. The decreasing rate is determined by the number of vertices.
\end{proof}
Now, we can establish theorem \ref{thm:2region}.
\begin{proof}[Proof of theorem \ref{thm:2region}]
The decreasing rate of the area of an enclosed region with $m$ edges is completely determined by $m$. We have
\begin{equation}\nonumber
\frac{da}{dt}=(m-6)\frac{\pi}{3}.
\end{equation}
If we perturb the network such that the area ratio between different regions changes, the area of different regions may not goes to zero at the same time. At the singularity, at most one enclosed region vanishes. Therefore, the tangent flow for the singularity has at most one enclosed region.
\end{proof}

\section{The F-functional}\label{sec:F_functional}
Now, we focus on the regular shrinkers with one enclosed region. To show some of such regular shrinkers are unstable, we need to use the $F$-functional and the entropy. Define the backward heat kernel $\Phi:\mathbb{R}^2\times(-\infty,0)\to\mathbb{R}$ by
\begin{equation}
    \Phi(x,t)=\frac{1}{\sqrt{-4\pi t}}e^\frac{|x|^2}{4t}.
\end{equation}
For $x_0\in\mathbb{R}^2$, $t_0\in(0,\infty)$, set $\Phi_{(x_0,t_0)}(x,t)=\Phi(x-x_0,t-t_0)$. Let $\Gamma$ be a network. The $F$-functional is defined as
\begin{equation}
    F_{x_0,t_0}(\Gamma)=\frac{1}{\sqrt{4\pi t_0}}\int_\Gamma e^{-\frac{|x-x_0|^2}{4t_0}}d\sigma=\int_\Gamma \Phi_{(x_0, t_0)}(x,0)d\sigma.
\end{equation}
The Huisken's monotonicity formula is first established by Huisken in \cite{H}. It is generalized to the network flow in \cite {MNPS}. If a family of open network $\Gamma_t$ flows according to the network flow, then
\begin{equation}
    \frac{d}{dt}\int_{\Gamma_t}\Phi_{(x_0,t_0)}(x,t)d\sigma=-\int_{\Gamma_t}\left|kN+\frac{(x-x_0)^\perp}{2(t_0-t)}\right|^2\Phi_{(x_0,t_0)}(x,t)d\sigma.
\end{equation}
This implies $F_{(x_0,t_0-t)}(\Gamma_t)=\int_{\Gamma_t}\Phi_{(x_0,t_0)}(x,t)d\sigma$ is nonincreasing. We have the following properties of the $F$-functional:
\begin{enumerate}
    \item Translation invariance: For any $y\in\mathbb{R}^2$, $F_{(0,t_0)}(\Gamma-y)=F_{(y,t_0)}(\Gamma)$. 
    \item Scaling invariance: For any $\alpha>0$, $F_{(0,\alpha^2 t_0)}(\alpha\Gamma)=F_{(0,t_0)}(\Gamma)$.
    \item Monotonicity: For all $t_1<t_2$, $F_{(x_0,t_0)}(\Gamma_{t_1})>F_{(x_0,t_0+(t_1-t_2))}(\Gamma_{t_2})$.
\end{enumerate}
The first two properties are just change of variable on $\Gamma$. From the translation and scaling invariance, we can reduce the study of $F$-functional at $(x_0,t_0)$ to the $F$-functional at $(0,1)$. The last property is obtained from Huisken's monotonicity formula.

The entropy $\lambda$ is defined by
\begin{equation}
    \lambda(\Gamma)=\sup_{x_0,t_0}F_{x_0,t_0}(\Gamma).
\end{equation}
The entropy is nonincreasing under network flow. It is invariant under scaling and rotation. Therefore, it is an important tool to determine whether the tangent flow of the perturbed network can flow back to the same singularity.

\section{The first variation of $F_{x_0, t_0}$ and regular shrinkers}\label{sec:1st_var}
In this section, we will derive the first variation formula of the $F$-functional. Since the $F$-functional can be defined not only on regular networks but on any networks, we derive in the general case that a network $\Gamma$ need not be regular.
\begin{defn}
$\Gamma_s$ is a variation of $\Gamma$ if $\Gamma_s$ is a one parameter family of embeddings $X_s:\Gamma\to\mathbb{R}^2$ with $X_0$ equal to the identity. The vector field $\frac{\partial X_s}{\partial s}|_{s=0}$ is the variation vector field.
\end{defn}

On a curve, we can reparametrize such that the variation vector field only has the normal component. However, when there are multi-junctions, we need to deal with the tangent component carefully. The following lemma is useful. We may consider the normal variation at the smooth part and the tangent variation at the endpoints.

\begin{lem}\label{lem:normal}
On a curve $\gamma$ from $P_1$ to $P_2$, for any variation vector field $V=v N+v^TT$, any function $f\in\mathbf{C}^1(\mathbb{R}^2)$, we have
\begin{equation}
    \left(\partial_s\int_\gamma fd\sigma\right)\bigg|_{s=0}=\int_\gamma v\left(\partial_N f-fk\right)d\sigma-\langle V,\hat{T}\rangle f(P_1)-\langle V,\hat{T}\rangle f(P_2),
\end{equation}
where we choose the tangent vector $\hat{T}(P_i)$ points towards the curve at the endpoints. Note that we cannot continuously define the tangent vector $\hat{T}$ from one end point to the other one.
\end{lem}
\begin{proof} Note that we have $\frac{d}{ds}d\sigma=(\partial_T v^T-v k)d\sigma$.
\begin{equation}
\begin{split}
    \partial_s\int_\gamma fd\sigma&=\int_\gamma\left(v\partial_N f+v^T \partial_T f+f(\partial_T v^T-v k)\right)d\sigma\\
    &=\int_\gamma\left(v(\partial_N f-kf)+\partial_T(v^T f)\right)d\sigma=\int_\gamma v(\partial_N f-kf)d\sigma - \sum_{1,2} \langle V,\hat{T} \rangle f(P_i).
\end{split}
\end{equation}
\end{proof}

Now, we can compute the first variation formula of the $F$-functional. Here, we introduce the following notation: For a function $f$ on $\Gamma$,
\begin{equation}
    \left[f\right]_{(x_0,t_0)}=\int_\Gamma f\Phi_{(x_0,t_0)}(x,0)d\sigma
\end{equation}
\begin{lem}[First variation formula]
Let $\Gamma_s$ be a variation of $\Gamma$ with the variation vector field $V$. If $x_s$ and $t_s$ are variations of $x_0$ and $t_0$ with $x_0'=y$ and $t_0'=h$, the first variation $\partial_s\left(F_{(x_s,t_s)}(\Gamma_s)\right)|_{s=0}$ is given by
\begin{equation}\label{eq:1st_var}
\begin{split}
    &-\left[v\left(k+\frac{\langle x-x_0,N\rangle}{2t_0}\right)\right]_{(x_0,t_0)}+\left[h\left(\frac{|x-x_0|^2}{4t_0^2}-\frac{1}{2t_0}\right)+\frac{\langle x-x_0,y\rangle}{2t_0}\right]_{(x_0,t_0)}\\
    &-\frac{1}{\sqrt{4\pi t_0}}\sum_i \left\langle  \sum_j\hat{T}_i^j,V(O_i)\right\rangle e^{-\frac{|O_i-x_0|^2}{4t_0}}.
\end{split}
\end{equation}
\end{lem}

\begin{proof} We want to differentiate
\begin{equation}
    F_{(x_s,t_s)}(\Gamma_s)=\frac{1}{\sqrt{4\pi t_s}}\int_{\Gamma_s} e^{-\frac{|x-x_s|^2}{4t_s}}d\sigma
\end{equation}
with respect to $s$. Let $\gamma$ be a curve connecting $O_{i_1}$ and $O_{i_2}$, we have
\begin{equation}
\begin{split}
    \partial_s&\left(F_{(x_s,t_s)}(\gamma_s)\right)|_{s=0}=\left(-\frac{h}{2t_0}\right)\frac{1}{\sqrt{4\pi t_0}}\int_{\gamma} e^{-\frac{|x-x_0|^2}{4t_0}}d\sigma\\
    &+\frac{1}{\sqrt{4\pi t_0}}\int_{\gamma} \left(h\frac{|x-x_0|^2}{4t_0^2}-\frac{\langle x-x_0,v N-y\rangle}{2t_0}-v k\right)e^{-\frac{|x-x_0|^2}{4t_0}}d\sigma\\
    &-\frac{1}{\sqrt{4\pi t_0}}\left( \langle \hat{T}(O_{i_1}),V(O_{i_1})\rangle e^{-\frac{|O_{i_1}-x_0|^2}{4t_0}}+\langle \hat{T}(O_{i_2}),V(O_{i_2})\rangle e^{-\frac{|O_{i_2}-x_0|^2}{4t_0}}\right).
\end{split}
\end{equation}
Summing over all curves yields the result.
\end{proof}

\begin{rmk}
Since the $F$-functional satisfies $F_{(x_0,\alpha^2t_0)}(\Gamma)=F_{(0,t_0)}(\alpha^{-1}(\Gamma-x_0))$, to study the critical point of the $F$-functional, we can reduce the problem to the case $(x_0,t_0)=(0,1)$. In this case, the first variational formula becomes
\begin{equation}
\begin{split}
    \partial_s(F_{(x_s,t_s)}(\Gamma_s))|_{s=0}=&-\left[v\left(k+\frac{\langle x,N\rangle}{2}\right)+h\left(\frac{|x|^2}{4}-\frac{1}{2}\right)+\frac{\langle x,y\rangle}{2}\right]_{(0,1)}\\
    &-\frac{1}{\sqrt{4\pi}}\sum_i \left\langle  \sum_j\hat{T}_i^j,V(O_i)\right\rangle e^{-\frac{|O_i|^2}{4}}.
\end{split}    
\end{equation}
\end{rmk}

\begin{defn}
A multi-junction $O_i$ is balanced if $\sum_j\hat{T}_i^j=0$. A general network $\Gamma$ is balanced if every multi-junction $O_i$ are balanced. Moreover, if $\Gamma$ also satisfies \begin{equation}
    k+\frac{\langle x,N\rangle}{2}=0
\end{equation}
at regular points, it is called a balanced shrinker.
\end{defn}
\begin{rmk}
If $O_i$ is a balanced triple-junction, we can deduce that the angle between the tangents $\hat{T}_i^j$ are $\frac{2\pi}{3}$. It satisfies the Herring condition at this multi-junction. Therefore, if all multi-junctions of a balanced network $\Gamma$ are triple junctions, we can deduce $\Gamma$ is regular.
\end{rmk}

\begin{thm}
The critical points of $F$ at $(0,1)$ are balanced shrinkers.
\end{thm}
\begin{proof}
Let $\Gamma$ be a critical point of the $F$-functional when $(x_0,t_0)=(0,1)$. From the first variation formula \eqref{eq:1st_var}, we can choose an arbitrary variation vector field $V$ to be supported on a single curve, hence
\begin{equation}
    k+\frac{\langle x,N\rangle}{2}=0.
\end{equation}
Again, since we can choose $V$ arbitrarily at the multi-junction $O_i$, the multi-junction must be balanced. This concludes the proof of the theorem.
\end{proof}

In the last theorem, it is showed that the critical point of the $F$-functional must be a balanced shrinker. Now, we establish the opposite. Here, we define the drift Laplacian $\mathcal{L}$. Since we focus on the variation at $(x_0,t_0)=(0,1)$, we use $[[f]]$ for $[f]_{(0,1)}$. We also use $\bar{\nabla}$ to denote the gradient in $\mathbb{R}^2$ and use $\nabla$ and $\Delta$ to denote the operations on $\Gamma$.
\begin{defn}
The drift laplacian $\mathcal{L}$ is defined as
\begin{equation}
    \mathcal{L}v=\Delta v-\frac{1}{2}\left\langle x, \nabla v\right\rangle=e^\frac{|x|^2}{4}\mathbf{div}\left(e^{-\frac{|x|^2}{4}}\nabla v\right).
\end{equation}
Note that $\Delta$ and $\nabla$ is calculated on $\Gamma$.
\end{defn}
\begin{lem}
If $\Gamma\subset \mathbb{R}^2$ is a balanced network, $v\in\mathbf{C}^1(\mathbb{R}^2)$, $w\in\mathbf{C}^2(\mathbb{R}^2)$, then
\begin{equation}
    \left[\left[v\mathcal{L}w\right]\right]=-\left[\left[\langle \nabla v,\nabla w\rangle\right]\right].
\end{equation}
\end{lem}
\begin{proof}
\begin{equation}\label{integ_by_parts}
    \int_\gamma (v\mathcal{L}w+\langle \nabla v,\nabla w\rangle) e^{-\frac{|x|^2}{4}}d\sigma=\int_\gamma \mathbf{div}(v\nabla w e^{-\frac{|x|^2}{4}})d\sigma=-\sum_{\partial\gamma} v\langle \nabla w,\hat{T}\rangle e^{-\frac{|x|^2}{4}}.
\end{equation}
Applying the divergence theorem, there will be boundary terms. The contribution of boundary term at a triple junction $O_i$ is 
\begin{equation}
    -e^{-\frac{|x|^2}{4}}\sum_jv\left\langle\left(\nabla w\right)_i^j,\hat{T}_i^j\right\rangle\Big|_{O_i}=-e^{-\frac{|O_i|^2}{4}}v(O_i)\left\langle\bar{\nabla}w,\sum_j\hat{T}_i^j\right\rangle=0.
\end{equation}
The last equality holds since $\Gamma$ is balanced.
\end{proof}

\begin{rmk}
If $f,g\in\mathbf{C}^2(\mathbb{R}^2)$, we have
\begin{equation}
    \left[\left[f\mathcal{L}(g)\right]\right]=-\left[\left[\langle \nabla f,\nabla g\rangle\right]\right]=\left[\left[\mathcal{L}(f)g\right]\right].
\end{equation}
The operator $\mathcal{L}$ is symmetric with respect to $e^{-\frac{|x|^2}{4}}d\sigma$.
\end{rmk}

\begin{lem}
If $\Gamma\subset\mathbb{R}^2$ satisfies $k+\frac{\langle x,N\rangle}{2}=0$ at all regular points, we have
\begin{equation}
\begin{split}
    \mathcal{L}x_i&=-\frac{1}{2}x_i,\\
    \mathcal{L}|x|^2&=2-|x|^2.
\end{split}
\end{equation}
\end{lem}

\begin{proof}
We have $\Delta x=kN$ on a curve, together with $k+\frac{\langle x,N\rangle}{2}=0$,
\begin{equation}
    \Delta x_i=\langle kN,e_i\rangle=-\frac{\langle x,N\rangle}{2}\langle N, \bar{\nabla} x_i\rangle=-\frac{1}{2}\langle x, \bar{\nabla} x_i\rangle+\frac{\langle x,T\rangle}{2}\langle T, \bar{\nabla} x_i\rangle.
\end{equation}
Note that $\frac{\langle x,T\rangle}{2}\langle T, \bar{\nabla} x_i\rangle=\frac{\langle x,\nabla x_i\rangle}{2}$. We obtain $\mathcal{L}x_i=-\frac{1}{2}x_i$. Now, use $\nabla|x|^2=2\langle x,T\rangle T$, we have
\begin{equation}
\begin{split}
    \Delta|x|^2&=2\langle\Delta x,x\rangle+2|\nabla x|^2=2\langle kN,x\rangle+2=2\cdot\left(-\frac{\langle x,N\rangle}{2}\right)\langle x,N\rangle\\
    &=-\langle x,N\rangle^2+2=2-|x|^2+\frac{1}{2}\langle x,\nabla|x|^2\rangle.
\end{split}
\end{equation}
\end{proof}

\begin{lem}\label{lem:bracket}
If $\Gamma$ is a balanced shrinker, then
\begin{equation}
\begin{split}
    \left[\left[|x|^2-2\right]\right]&=\left[\left[|x|^4-12+16k^2\right]\right]=0,\\
    \left[\left[x\right]\right]&=\left[\left[x|x|^2\right]\right]=0.
\end{split}
\end{equation}
Also, for any vector $W\in\mathbb{R}^2$, we have
\begin{equation}
    \left[\left[\langle x,W\rangle^2\right]\right]=\left[\left[2\langle W,T\rangle^2\right]\right].
\end{equation}
\end{lem}

\begin{proof}
From the identity regarding $\mathcal{L}$, we have
\begin{equation}
\begin{split}
    \left[\left[|x|^2-2\right]\right]&=\left[\left[-1\cdot\mathcal{L}(|x|^2)\right]\right]=-\left[\left[\mathcal{L}(1)|x|^2\right]\right]=0,\\
    \left[\left[x_i\right]\right]&=\left[\left[(-2)\cdot\frac{-x_i}{2}\right]\right]=\left[\left[(-2)\cdot\mathcal{L}(x_i)\right]\right]=\left[\left[\mathcal{L}(-2)x_i\right]\right]=0,\\
    \left[\left[-\frac{1}{2}x_i|x|^2\right]\right]&=\left[\left[\mathcal{L}(x_i)|x|^2\right]\right]=\left[\left[x_i\mathcal{L}(|x|^2)\right]\right]=\left[\left[x_i(2-|x|^2)\right]\right]=-\left[\left[x_i|x|^2\right]\right],
\end{split}
\end{equation}
where the last equation is equivalent to $\left[\left[ x_i|x|^2\right]\right]=0$.

We also have
\begin{equation}
    \left[\left[\mathcal{L}(|x|^2)|x|^2\right]\right]=\left[\left[ (2-|x|^2)|x|^2\right]\right].
\end{equation}
On the other hand, use integration by parts, we have
\begin{equation}
\begin{split}
    \left[\left[\mathcal{L}(|x|^2)|x|^2\right]\right]&=-\left[\left[ \langle \nabla|x|^2,\nabla|x|^2\rangle \right]\right]=-\left[\left[ |2\langle x,T\rangle|^2 \right]\right]\\
    &=-\left[\left[ 4|x|^2-4\langle x,N\rangle^2 \right]\right]=-\left[\left[ 4|x|^2-16k^2 \right]\right].
\end{split}
\end{equation}
We have $\left[\left[|x|^4-6|x|^2+16k^2\right]\right]=0$. Combining it with $\left[\left[6(|x|^2-2)\right]\right]=0$ yields the desired result.

For the last part,
\begin{equation}
\begin{split}
    \left[\left[ -\frac{1}{2}\langle x,W\rangle^2 \right]\right]&=\left[\left[ \mathcal{L}(\langle x,W\rangle)\langle x,W\rangle \right]\right]=-\left[\left[ |\nabla\langle x,W\rangle|^2\right]\right]=-\left[\left[ \langle W,T\rangle^2\right]\right].
\end{split}
\end{equation}
\end{proof}

Now, we can characterize the critical points of the $F$-functional.
\begin{thm}
If a network $\Gamma$ is balanced and satisfies $k+\frac{\langle x-x_0,N\rangle}{2t_0}=0$, it is a critical point of the $F$-functional at $(x_0,t_0)$.
\end{thm}
\begin{proof}
If $\Gamma$ is the time $t_0$ slice of a shrinking solution centered at $x_0$. The term in \eqref{eq:1st_var} involving $V$ and the normal part $v$ vanishes immediately. We need to show the terms involving $h$ and $y$ also vanishes. Without loss of generality, we can assume $x_0=0$ and $t_0=1$. The term involving $h$ is
\begin{equation}
    \left[\left[-2+|x|^2\right]\right]\frac{h}{4},
\end{equation}
and the term involving $y$ is
\begin{equation}
    \left\langle\left[\left[\frac{x}{2}\right]\right],y\right\rangle.
\end{equation}
From the previous lemma, these two terms vanish. For the general $(x_0,t_0)$, similar equations are derived by scaling and translating the equations in the previous lemma.
\end{proof}

\section{Relation between Variation vector field and the normal data}
From now on, we focus on the regular networks. Assume all multi-junctions satisfy the Herring condition.

We want to use normal component to describe the variation. Choose the normal vector field $N$ on each curve $\gamma$. At a triple junction $O_i$, let $\gamma_i^j$ be a curve ending at $O_i$. Locally define $\hat{T}_i^j$ to be the tangent vector field pointing towards the curve from $O_i$. Let $\mathcal{R}$ be the linear transform in $\mathbb{R}^2$ which rotates a vector by $\frac{\pi}{2}$ in the counterclockwise direction. At each multi-junction $O_i$, define the signature $\eta_i^j=\pm1$ so that the normal vector $N_i^j=\eta_i^j\mathcal{R}(\hat{T}_i^j)$. We have
\begin{equation}
    \sum_j\eta_i^jN_i^j=\sum_j\mathcal{R}(\hat{T}_i^j)=\mathcal{R}(\sum_j\hat{T}_i^j)=0.
\end{equation}
For a variation vector field $V$ on $\Gamma$, we use lowercase $v$ to denote the normal component $\langle V, N\rangle$. At a triple junction $O_i$, we use the notation $v_i^j=\langle V,N_i^j\rangle$. It behaves like a multi-valued function at $O_i$. If $v$ is the normal data of a vector field $V$, we have $\sum_j\eta_i^jv_i^j=0$ at each multi-junction $O_i$. Similarly for any function $w$ which si continuous of the smooth part of $\Gamma$, define
\begin{equation}\nonumber
    w_i^j(O_i)=\lim_{P\in\gamma_i^j,P\to O_i}w(P),
\end{equation}
this can be think of as the value of the function seen from $\gamma_i^j$.

Now, we consider the function space $\mathfrak{V}(\Gamma)=\{v\in \mathrm{C}^1(\Gamma)|\sum_j\eta_i^jv_i^j=0\}$. For all variation vector field $V$, we have $v=\langle V,N\rangle\in\mathfrak{V}(\Gamma)$. Now, we establish the opposite: each function $v\in\mathfrak{V}(\Gamma)$ can be obtained from a variation.

\begin{prop}
For any $v\in\mathfrak{V}(\Gamma)$, there is a variation of $\Gamma$ such that the variation vector field $V$ satisfies $\langle V,N\rangle=v$.
\end{prop}
\begin{proof}
At a triple junction $O_i$, without loss of generality, we assume the curves $\gamma_i^j$ meeting at $O_i$ are labeled in the counterclockwise direction. Let $\alpha_i^j$ be a smooth cutoff function on $\gamma_i^j$ such that $\alpha_i^j(O_i)=1$ and it is supported in a small neighborhood of $O_i$.

On $\gamma_i^j$, near $O_i$ let the vector field $V$ be defined as
\begin{equation}
    V=v_i^jN_i^j+\frac{1}{\sqrt{3}}(-\eta_i^{j+1}v_i^{j+1}+\eta_i^{j+2}v_i^{j+2})\alpha_i^j\hat{T}_i^j,
\end{equation}
where the indices $j+1$, $j+2$ are module by 3. This definition is continuous on every curves. We have to check this is also continuous at the multi-junctions. At $O_i$, from $\gamma_i^j$ side, we have
\begin{equation}
\begin{split}
    V_i^j&(O_i)=v_i^jN_i^j+\frac{1}{\sqrt{3}}(-\eta_i^{j+1}v_i^{j+1}+\eta_i^{j+2}v_i^{j+2})\hat{T}_i^j\\
    =&v_i^jN_i^j+\frac{1}{\sqrt{3}}(-\eta_i^{j+1}v_i^{j+1}+\eta_i^{j+2}v_i^{j+2})\frac{1}{\sqrt{3}}(\eta_i^{j+2}N_i^{j+2}-\eta_i^{j+1}N_i^{j+1})\\
    =&v_i^jN_i^j+\frac{1}{3}(\eta_i^{j+2}v_i^{j+2}-\eta_i^{j+1}v_i^{j+1})(\eta_i^{j+2}N_i^{j+2})+\frac{1}{3}(\eta_i^{j+1}v_i^{j+1}-\eta_i^{j+2}v_i^{j+2})(\eta_i^{j+1}N_i^{j+1})\\
    =&\frac{2}{3}v_i^jN_i^j+\frac{1}{3}(\eta_i^{j+1}v_i^{j+1}+\eta_i^{j+2}v_i^{j+2})(\eta_i^{j+2}N_i^{j+2}+\eta_i^{j+1}N_i^{j+1})\\
    &+\frac{1}{3}(\eta_i^{j+2}v_i^{j+2}-\eta_i^{j+1}v_i^{j+1})(\eta_i^{j+2}N_i^{j+2})+\frac{1}{3}(\eta_i^{j+1}v_i^{j+1}-\eta_i^{j+2}v_i^{j+2})(\eta_i^{j+1}N_i^{j+1})\\
    =&\frac{2}{3}v_i^jN_i^j+\frac{2}{3}v_i^{j+1}N_i^{j+1}+\frac{2}{3}v_i^{j+2}N_i^{j+2}.\\
\end{split}
\end{equation}
This is independent of $j$, therefore, the vector field $V$ is continuous even at the triple-junctions.
\end{proof}

\begin{rmk}
The above construction only works for a balanced triple junction. For other balanced multi-junctions, we may not construct the variational vector field $V$ even if $v$ satisfies $\sum_j\eta_i^jv_i^j$ at all multi-junctions $O_i$.
\end{rmk}

\section{The second variation of $F_{x_0,t_0}$ and the operator $L_{x_0,t_0}$}\label{sec:2nd_var}

\begin{lem}
If $V=v^TT+vN$ is a variation vector field, we have 
\begin{equation}
\begin{split}
    \partial_s N&=-\nabla v-v^TkT,\\
    \partial_s T&=(\partial_T v+v^Tk)N,\\
    \partial_s k&=\Delta v +v k^2+v^T\partial_T k.\\
\end{split}
\end{equation}
\end{lem}

\begin{proof}
First, let $\gamma_s(t)$ be a paratrization of the curve such that $t$ is the arclength for $\gamma_0(t)$. Differentiate the equation $\langle N,\partial_t\gamma\rangle=0$ with respect to $s$, we have
\begin{equation}
\begin{split}
    \langle \partial_s N,\partial_t\gamma\rangle|_{s=0}&=-\langle N,\partial_t\partial_s \gamma\rangle|_{s=0}=-\langle N,\partial_t(v^TT+v N)\rangle\\
    &=-\langle N,\partial_t v^TT+v^TkN+\partial_tv N-v kT\rangle=-\partial_t v-v^Tk.
\end{split}
\end{equation}
Therefore, $\partial_s N=-\nabla v-v^TkT$. From $\langle \partial_sN,T\rangle=-\langle N,\partial_s T\rangle$, we can deduce $\partial_sT=(\partial_T v+v^Tk)N$.

Let $m$ be the speed of the curve $\gamma_s(t)$, $\partial_t \gamma=mT$. Differentiate the equation $m^2=|\partial_t \gamma|^2$ with respect to $s$ yields
\begin{equation}
\begin{split}
    2m\partial_s m&=\partial_s(|\partial_t\gamma|^2)=2\langle \partial_t\gamma,\partial_t(v N+v^TT)\rangle\\
    &=2\langle mT,v(-kmT)+\partial_tvN+\partial_tv^TT+v^TkmN\rangle.
\end{split}
\end{equation}
Therefore, at $s=0$, we have $m=1$ and $\partial_s m=-v k+\partial_T v^T$. Now, from $\partial_t^2\gamma=\partial_t mT+m^2kN$,
\begin{equation}
\begin{split}
    \langle N,\partial_s\partial_t^2\gamma\rangle|_{s=0}&=\langle N,\partial_s(\partial_t mT+m^2kN)\rangle|_{s=0}\\
    &=\langle N,\partial_s\partial_t mT+\partial_t m\partial_s T+(2m\partial_smk+m^2\partial_sk)N+m^2k\partial_sN\rangle|_{s=0}\\
    &=2(-v k+\partial_T v^T)k+\partial_sk.
\end{split}
\end{equation}
Note that at $s=0$, the variation vector field is $\partial_s\gamma=V=vN+v^TT$.
\begin{equation}
\begin{split}
    \partial_t V=&\partial_t vN-vkT+\partial_t v^TT+v^TkN=(-vk+\partial_T v^T)T+(\partial_T v+v^Tk)N,\\
    \partial_t^2 V=&(-\partial_Tvk-v\partial_Tk+\Delta v^T)T+(\Delta v+\partial_T v^Tk+v^T\partial_T k)N\\
    &+(-vk+\partial_T v^T)kN-(\partial_T v+v^Tk)kT.
\end{split}
\end{equation}
We also have
\begin{equation}
\begin{split}
    \langle N,\partial_s\partial_t^2\gamma\rangle|_{s=0}&=\langle N,\partial_t^2 V\rangle=\Delta v+\partial_T v^Tk+v^T\partial_T k-vk^2+\partial_T v^Tk.
\end{split}
\end{equation}
Therefore,
\begin{equation}
    2(-v k+\partial_T v^T)k+\partial_sk=\Delta v -v k^2+2\partial_T v^Tk+v^T\partial_T k
\end{equation}
and we obtain the desired result.
\end{proof}

\begin{rmk}
In the special case that $V=vN$ is a normal variation vector field, we have 
\begin{equation}
\begin{split}
    \partial_s N&=-\nabla v,\\
    \partial_s k&=\Delta v+vk^2.\\
\end{split}
\end{equation}
\end{rmk}
If $\Gamma$ is the $-t_0$ time slice of a self-similarly shrinking solution centered at $x_0$, then $\frac{(\Gamma-x_0)}{\sqrt{t_0}}$ is the $-1$ time slice of a self-similarly shrinking solution centered at the origin. Without loss of generality, we only need to consider the case that $\Gamma$ is a regular shrinker.

\begin{thm}[Second variation formula]
Let $\Gamma$ be a regular shrinker. $\Gamma_s$, $x_s$, $t_s$, $s=(s_1,s_2)$ are a two-parameter variation with $\Gamma_0=\Gamma$, $x_0=0$, $t_0=1$. The variation directions are given by $\partial_{s_1}\Gamma_s|_{s=0}=V$, $\partial_{s_1}x_s|_{s=0}=y$, $\partial_{s_1}t_s|_{t=0}=h$, and $\partial_{s_2}\Gamma_s|_{s=0}=W$, $\partial_{s_2}x_s|_{s=0}=z$, $\partial_{s_2}t_s|_{t=0}=l$. We have
\begin{equation}\label{eq:2nd_var}
\begin{split}
    \partial_{s_2}\partial_{s_1}&(F_{(x_s,t_s)}(\Gamma_s))|_{s=0}\\
    =&\left[\left[-vLw+\frac{v\langle z,N\rangle+w\langle y,N\rangle}{2}-(vl+wh)k\right]\right]\\
    &+\left[\left[ -hlk^2-\frac{\langle y,z\rangle}{2}+\frac{\langle x,y\rangle\langle x,z\rangle}{4}\right]\right]-\frac{1}{\sqrt{4\pi}}\sum_{i,j}v_i^j\left\langle \nabla\left(we^{-\frac{|x|^2}{4}}\right)_i^j ,\hat{T}_i^j\right\rangle.
\end{split}
\end{equation}
where the lower case $v$, $w$, denotes $\langle V,n\rangle$, $\langle W,n\rangle$ respectively, $Lw=\mathcal{L}w+k^2w+\frac{w}{2}$.
\end{thm}
\begin{rmk}
The terms other than the multi-junction term are consistent with Colding and Minicozzi's work \cite{CM}.
\end{rmk}

\begin{proof}
From the first variation formula, $\partial_{s_1}\left(F_{(x_s,t_s)}(\Gamma_s)\right)$ is
\begin{equation}
\begin{split}
    &-\left[v\left(k+\frac{\langle x-x_s,N\rangle}{2t_s}\right)\right]_{(x_s,t_s)}+\left[h\left(-\frac{1}{2t_s}+\frac{|x-x_s|^2}{4t_s^2}\right)+\frac{\langle x-x_s,y\rangle}{2t_s}\right]_{(x_s,t_s)}\\
    &-\frac{1}{\sqrt{4\pi t_s}}\sum_i \left\langle  \sum_j\hat{T}_i^j,V\right\rangle e^{-\frac{|x-x_s|^2}{4t_s}}.
\end{split}
\end{equation}
First noted that each term has a factor of $\frac{1}{\sqrt{4\pi t_s}}$. If the differentiation applies to this term, it will result in a term that is multiple for the first variation. This term vanishes since every regular shrinker $\Gamma$ is a critical point of $F$-functional at $(0,1)$. We will deal with each term in the first variation separately.

For the first term, if the differentiation acts on terms other than $k+\frac{\langle x-x_s,N\rangle}{2t_s}$, this part will remain in the result and yields 0. We only need to differentiate $k+\frac{\langle x-x_s,N\rangle}{2t_s}$ with respect to $s_2$. Using lemma \ref{lem:normal}, we obtain
\begin{equation}
\begin{split}
    -\partial_{s_2}&\left[v\left(k+\frac{\langle x-x_s,N\rangle}{2t_s}\right)\right]_{(x_s,t_s)}\bigg|_{s=0}=-\left[v\partial_{s_2}\left(k+\frac{\langle x-x_s,N\rangle}{2t_s}\right)\right]_{(x_s,t_s)}\bigg|_{s=0}\\
    &=-\left[\left[ v\left(\Delta w+w k^2+w^T\partial_Tk+\frac{\langle wN-z,N\rangle}{2}-\frac{l\langle x,N\rangle}{2}-\frac{\langle x,\nabla w+w^TkT\rangle}{2}\right)\right]\right]\\
    &=-\left[\left[ v\left(Lw-\frac{\langle z,N\rangle}{2}+lk\right)\right]\right]-\left[\left[ vw^T\left(\partial_Tk-\frac{\langle x,kT\rangle}{2}\right)\right]\right].\\
\end{split}
\end{equation}
Note that since $\partial_Tk-\frac{\langle x,kT\rangle}{2}=\partial_T(k+\frac{\langle x,N\rangle}{2})=0$, the term involve tangential differentiation vanishes.

For the second term,
\begin{equation}
\begin{split}
    \partial_{s_2}&\left(h\left(-\frac{1}{2t_s}+\frac{|x-x_s|^2}{4t_s^2}\right)+\frac{\langle x-x_s,y\rangle}{2t_s}\right)\bigg|_{s=0}\\
    &=h\left(l\left(\frac{1}{2}-\frac{|x|^2}{2}\right)+\frac{\langle x, wN+w^TT-z\rangle}{2}\right)-\frac{\langle x,y\rangle}{2}l+\frac{\langle wN+w^TT-z,y\rangle}{2}.
\end{split}
\end{equation}
Since $\int_\Gamma (|x|^2-2)e^{-\frac{|x|^2}{4}}d\sigma=0$, $\int_\Gamma xe^{-\frac{|x|^2}{4}}d\sigma=0$, and $k+\frac{\langle x,N\rangle}{2}=0$, the contribution of this part is
\begin{equation}
\begin{split}
    \frac{1}{\sqrt{4\pi}}\int_\Gamma \left(-\frac{hl}{2}-khw+\frac{w\langle N,y\rangle-\langle y,z\rangle}{2}+w^T\left(h\frac{\langle x,T\rangle}{2}+\frac{\langle T,y\rangle}{2}\right)\right)e^{-\frac{|x|^2}{4}}d\sigma.
\end{split}
\end{equation}
If the differentiation applies to $e^{-\frac{|x-x_s|^2}{4t_s}}d\sigma$, we have
\begin{equation}
\begin{split}
    \partial_{s_2}&\left(e^{-\frac{|x-x_s|^2}{4t_s}}d\sigma\right)\bigg|_{s=0}=\left(\frac{|x|^2}{4}l-\frac{\langle x,wN+w^TT-z\rangle}{2}-kw+\partial_Tw^T\right)e^{-\frac{|x|^2}{4}}d\sigma\\
    &=\left(\frac{|x|^2}{4}l+\frac{\langle x,z\rangle}{2}-\left(k+\frac{\langle x,N\rangle}{2}\right)w+\partial_Tw^T-w^T\frac{\langle x,T\rangle}{2}\right)e^{-\frac{|x|^2}{4}}d\sigma\\
    &=\left(\frac{|x|^2}{4}l+\frac{\langle x,z\rangle}{2}+\partial_Tw^T-w^T\frac{\langle x,T\rangle}{2}\right)e^{-\frac{|x|^2}{4}}d\sigma.\\
\end{split}
\end{equation}
 Combine the terms together, the second term is 
\begin{equation}
\begin{split}
    &\left[\left[-\frac{hl}{2}-khw+\frac{w\langle N,y\rangle-\langle y,z\rangle}{2}\right]\right]+\left[\left[w^T\left(h\frac{\langle x,T\rangle}{2}+\frac{\langle T,y\rangle}{2}\right)\right]\right]\\
    &+\left[\left[\left(h\left(-\frac{1}{2}+\frac{|x|^2}{4}\right)+\frac{\langle x,y\rangle}{2}\right)\left(\frac{|x|^2}{4}l+\frac{\langle x,z\rangle}{2}+\partial_Tw^T-w^T\frac{\langle x,T\rangle}{2}\right)\right]\right]\\
    =&\left[\left[hl\left(\frac{|x|^4}{16}-\frac{3}{4}\right)-khw+\frac{w\langle N,y\rangle-\langle y,z\rangle}{2}+\frac{\langle x,y\rangle\langle x,z\rangle}{4}\right]\right]\\
    &+\left[\left[\partial_T\left(w^T\left(h\frac{\langle x,T\rangle}{2}+\frac{\langle T,y\rangle}{2}\right)\right)-w^T\left(h\left(-\frac{1}{2}+\frac{|x|^2}{4}\right)+\frac{\langle x,y\rangle}{2}\right)\frac{\langle x,T\rangle}{2}\right]\right],\\
\end{split}
\end{equation}
where we use $[[x]]=[[x|x|^2]]=0$ and $[[|x|^2-2]]=0$ to simplify the equation. Use divergent theorem on the second term yields
\begin{equation}
    -\sum_i\sum_j\left\langle\hat{T}_i^j,W\right\rangle\left[h\left(-\frac{1}{2}+\frac{|O_i|^2}{4}\right)+\frac{\langle O_i, y\rangle}{2}\right]e^{-\frac{|O_i|^2}{4}}.
\end{equation}
This term vanishes since $\Gamma$ is regular, $\sum_j\langle\hat{T}_i^j,W(O_i)\rangle=0$ for all $i$.

For the third term, we need the differentiation apply to $\sum_j\hat{T}_i^j$, otherwise this term remains and it vanishes when $s=0$.
\begin{equation}
    \partial_{s_2}\left(\sum_j \hat{T}_i^j\right)\bigg|_{s=0}=\sum_j\left((\partial_{\hat{T}}w)_i^j+k_i^j\langle W,\hat{T}_i^j\rangle\right)N_i^j.
\end{equation}
The third term becomes
\begin{equation}
-\partial_{s_2}\sum_i \left\langle  \sum_j\hat{T}_i^j,V(O_i)\right\rangle e^{-\frac{|O_i-x_s|^2}{4t_s}}\bigg|_{s=0}=-\sum_{i,j}\left((\partial_{\hat{T}} w)_i^j+k_i^j\left\langle W,\hat{T}_i^j\right\rangle\right)v_i^j e^{-\frac{|O_i|^2}{4}}.
\end{equation}
At any multi-junction $O_i$, since $\sum_j\eta_i^jv_i^j=0$, we have
\begin{equation}
\begin{split}
    0&=\sum_j\eta_i^jv_i^j\frac{\langle\mathcal{R}(O_i),W\rangle}{2}=\sum_j\eta_i^jv_i^j\frac{\langle \mathcal{R}(O_i),N_i^j\rangle\langle W,N_i^j\rangle+\mathcal{R}(O_i),\hat{T}_i^j\rangle\langle W,\hat{T}_i^j\rangle}{2}\\
    &=\sum_j\eta_i^jv_i^j\frac{\langle \mathcal{R}(O_i),\eta_i^j\mathcal{R}(\hat{T}_i^j)\rangle\langle W,N_i^j\rangle+\langle\mathcal{R}(O_i),-\eta_i^j\mathcal{R}(N_i^j)\rangle\langle W,\hat{T}_i^j\rangle}{2}\\
    &=\sum_jv_i^j\frac{\langle O_i,\hat{T}_i^j\rangle\langle W,N_i^j\rangle-\langle O_i,N_i^j\rangle\langle W,\hat{T}_i^j\rangle}{2}=\sum_jv_i^j\left(\frac{\langle O_i,\hat{T}_i^j\rangle w_i^j}{2}-k_i^j\langle W,\hat{T}_i^j\rangle\right).\\
\end{split}
\end{equation}
Therefore, 
\begin{equation}
\begin{split}
-\partial_{s_2}&\sum_i \left\langle  \sum_j\hat{T}_i^j,V\right\rangle e^{-\frac{|x-x_s|^2}{4t_s}}\bigg|_{s=0}=-\sum_{i,j}\left((\partial_{\hat{T}} w)_i^j+k_i^j\left\langle W,\hat{T}_i^j\right\rangle\right)v_i^j e^{-\frac{|O_i|^2}{4}}\\
    &=-\sum_{i,j}\left(\partial_{\hat{T}} w_i^j-\frac{\langle O_i, \hat{T}_i^j\rangle}{2} w_i^j\right)v_i^j e^{-\frac{|O_i|^2}{4}}=-\sum_{i,j}v_i^j\left\langle \nabla\left(we^{-\frac{|x|^2}{4}}\right)_i^j ,\hat{T}_i^j\right\rangle.
\end{split}
\end{equation}
\end{proof}

Since the bilinear form arises from the second derivative, is has the following property.  
\begin{lem}
The second variation is symmetric.
\end{lem}
\begin{proof}
Since $\sum_{i,j}v_i^j\langle \nabla(we^{-\frac{|x|^2}{4}})_i^j ,\hat{T}_i^j\rangle=\sum_{i,j}((\partial_{\hat{T}} w)_i^j-\frac{1}{2}\langle O_i, \hat{T}_i^j\rangle w_i^j)v_i^j e^{-\frac{|O_i|^2}{4}}$, the terms which appear to be asymmetric is $\int_{\Gamma} -v\mathcal{L}we^{-\frac{|x|^2}{4}}d\sigma-\sum_{i,j} v_i^j(\partial_{\hat{T}}w)_i^j  e^{-\frac{|O_i|^2}{4}}$. We can use integration by parts on each curve $\gamma$ with boundary $O_{i_1}$ and $O_{i_2}$. From the definition of the drift laplacian $\mathcal{L}$, $\mathcal{L}g=e^\frac{|x|^2}{4}\mathbf{div}(e^{-\frac{|x|^2}{4}}\nabla g)$, we have
\begin{equation}
\begin{split}
    -\int_{\gamma}v\mathcal{L}we^{-\frac{|x|^2}{4}}d\sigma&=-\int_{\gamma}\mathbf{div}\left(ve^{-\frac{|x|^2}{4}}\nabla w\right)d\sigma+\int_{\gamma} \langle\nabla v,\nabla w\rangle e^{-\frac{|x|^2}{4}}d\sigma\\
    &=ve^{-\frac{|x|^2}{4}}\partial_{\hat{T}} w(O_{i_1})+ve^{-\frac{|x|^2}{4}}\partial_{\hat{T}} w(O_{i_2})+\int_{\gamma} \langle\nabla v,\nabla w\rangle e^{-\frac{|x|^2}{4}}d\sigma
\end{split}
\end{equation}
If we sum the equation over all curves $\gamma$, we obtain
\begin{equation}
\begin{split}
    -\int_{\Gamma}&v\mathcal{L}we^{-\frac{|x|^2}{4}}d\sigma-\sum_{i,j}v_i^j\left(\partial_{\hat{T}}w\right)_i^j e^{-\frac{|O_i|^2}{4}}=\int_{\Gamma} \langle\nabla v,\nabla w\rangle e^{-\frac{|x|^2}{4}}d\sigma.
\end{split}
\end{equation}
Therefore, the second variation formula is symmetric.
\end{proof}

\section{The eigenfunction problem from the 2nd variation formula}\label{sec:eigunfunction}
From the second variation formula, we only need the normal data for a variation. Hence, we consider the variations described by the function space $\mathfrak{V}(\Gamma)=\{v\in \mathrm{C}^1(\Gamma)|\sum_j\eta_i^jv_i^j=0\}$. Now, we focus on the bilinear form induced by the second variation formula, let $(v,y,h)$, $(w,z,l)$ be two variations, define
\begin{equation}
\begin{split}
    [(v,y,h)&,(w,z,l)]=\left[\left[ -vLw+\frac{v\langle z,n\rangle+w\langle y,n\rangle}{2}-(vl+wh)k\right]\right]\\
    &+\left[\left[-hlk^2-\frac{\langle y,z\rangle}{2}+\frac{\langle x,y\rangle\langle x,z\rangle}{4}\right]\right]-\frac{1}{\sqrt{4\pi}}\sum_{i,j}v_i^j\left\langle\nabla\left(we^{-\frac{|x|^2}{4}}\right)_i^j,\hat{T}_i^j\right\rangle.
\end{split}
\end{equation}
Also, we consider the variation on the network only.
\begin{equation}\label{eq:bilinear_form}
\begin{split}
    [v,w]&=\left[\left[ -vLw\right]\right]-\frac{1}{\sqrt{4\pi}}\sum_{i,j} v_i^j\left\langle\nabla\left(we^{-\frac{|x|^2}{4}}\right)_i^j,\hat{T}_i^j\right\rangle\\
    &=\left[\left[\left\langle \nabla v,\nabla w\right\rangle-\left(k^2+\frac{1}{2}\right)vw\right]\right]+\frac{1}{2\sqrt{4\pi}}\sum_{i,j} v_i^jw_i^j\left\langle O_i,\hat{T}_i^j\right\rangle e^{-\frac{|O_i|^2}{4}}.
\end{split}
\end{equation}

On a balanced shrinker $\Gamma$, let the weighted Sobolev space $\hat{W}^{1,2}=\{v|\sum_j\sigma_i^jv_i^j=0$ for all multi-junction $O_i$, $\int_\Gamma (v^2+|\nabla v|^2)e^{-\frac{|x|^2}{4}}d\sigma<\infty\}$.

Now, define the energy functional on $\hat{W}^{1,2}$ by
\begin{equation}
    E(v)=\frac{\sqrt{4\pi}[v,v]}{\int_\Gamma v^2e^{-\frac{|x|^2}{4}}d\sigma}.
\end{equation}

Consider the critical points of $E(v)$. Since $[\cdot,\cdot]$ is symmetric, when $0=\frac{d}{dt}E(v+tw)|_{t=0}$, we have
\begin{equation}
    0=2\left(\sqrt{4\pi}[v,w]-E(v)\int_\Gamma vwe^{-\frac{|x|^2}{4}}d\sigma\right).
\end{equation}
Therefore, if $v$ is a critical point of $E(v)$, 
\begin{equation}
\begin{split}
    0&=\int_\Gamma w\left(-Lv-E(v)v\right)e^{-\frac{|x|^2}{4}}d\sigma-\sum_{i,j} w_i^j\left\langle\nabla\left(ve^{-\frac{|x|^2}{4}}\right)_i^j,\hat{T}_i^j\right\rangle.
\end{split}
\end{equation}

Since $w\in\hat{W}^{1,2}$ is arbitrary, choose $w$ to be supported away from the multi-points. We can deduce
\begin{equation}
    -Lv=E(v)v
\end{equation}
on each curve. It means that $v$ is an eigenfunction of $-L$ with eigenvalue $E(v)$.

The term involving the multi-point $O_i$ is $\sum_j w_i^j\langle\nabla (ve^{-\frac{|x|^2}{4}})_i^j,\hat{T}_i^j\rangle$. Since there are three curves meeting at a multi-point, we can deduce that for all $j$, $\eta_i^j\langle\nabla (ve^{-\frac{|x|^2}{4}})_i^j,\hat{T}_i^j\rangle$ are the same. We obtain an eigenfunction problem

\begin{equation}\label{eigenprob}
\displaystyle \begin{cases} 
-Lv=\lambda v, & \text{on $\gamma_i$'s} \\  
\sum_j \eta_i^j v_i^j=0, & \text{at $O_i$'s} \\  
\eta_i^j\left\langle\nabla \left(ve^{-\frac{|x|^2}{4}}\right)_i^j,\hat{T}_i^j\right\rangle \text{ are the same for all $j$}. & \text{at any $O_i$}  
 \end{cases}
\end{equation}

Some of the eigenfunctions of $L$ play important geometric roles, as described in the following theorem. 
\begin{thm}
The function $k$ is an eigenfunction of $L$. For each $V\in\mathbb{R}^2$, $\langle V,N\rangle$ is an eigenfunction of $L$. They satisfy
\begin{equation}
\begin{split}
    Lk&=k,\\
    L\langle V,N\rangle&=\frac{1}{2}\langle V,N\rangle.\\
\end{split}
\end{equation}
\end{thm}

\begin{proof}
For the function $k=-\frac{\langle x,N\rangle}{2}$, we have 
\begin{equation}
\begin{split}
\Delta k&=-\mathbf{div}\left(\frac{\langle x,-kT\rangle}{2}T\right)=-\frac{1}{2}\left(-k\langle T,T\rangle+\langle x,-\nabla_TkT-k^2N\rangle\right)\\
&=\frac{1}{2}\left(k+\langle x,\nabla k\rangle+k^2\langle x,N\rangle\right)=\frac{1}{2}k+\frac{1}{2}\langle x,\nabla k\rangle-k^3.
\end{split}
\end{equation}
Therefore, $Lk=k$. We need to check if it satisfies the boundary condition. At $O_i$, we have $k_i^j=-\frac{\langle x,N_i^j\rangle}{2}$ and $\sum_j\eta_i^jk_i^j=-\frac{1}{2}\langle x,\sum_j\eta_i^jN_i^j\rangle=0$. For each $j$, since on an Abresch and Langer curve, $ke^{-\frac{|x|^2}{4}}$ is a constant, we have $\nabla(ke^{-\frac{|x|^2}{4}})_i^j=0$ for each $j$. Therefore, $\eta_i^j\langle\nabla (v_i^je^{-\frac{|x|^2}{4}}),\hat{T}_i^j\rangle=0$ and it is independent of $j$.

For the function $\langle V,N\rangle$, we have 
\begin{equation}
\begin{split}
\Delta \langle V,N\rangle&=\mathbf{div}\left(\langle V,-kT\rangle T\right)=-\langle V,\partial_TkT+k^2N\rangle\\
&=\frac{\langle x,-kT\rangle}{2}\langle V,T\rangle-k^2\langle V,N\rangle=\frac{1}{2}\left\langle x,\nabla\langle V,N\rangle\right\rangle-k^2\langle V,N\rangle.
\end{split}
\end{equation}
Hence, we have $L\langle V,N\rangle=\frac{1}{2}\langle V,N\rangle$. Again, we need to check the boundary condition. The summation term at a multi-junction $O_i$ is given by $\sum_j\eta_i^j\langle V,N_i^j\rangle=\langle V,\sum_j\eta_i^jN_i^j\rangle=0$. 
\begin{equation}
    \partial_T \langle V,N\rangle=\langle V,-kT \rangle=-k\langle V,T\rangle.
\end{equation}
Hence, for each $j$, let $v=\langle V,N\rangle$,
\begin{equation}\label{eq:f_2}
\begin{split}
\eta_i^j\langle& \nabla (ve^{-\frac{|x|^2}{4}}),\hat{T}\rangle_i^j=\eta_i^j\left(-k\langle V,\hat{T}\rangle e^{-\frac{|x|^2}{4}}+\langle \nabla e^{-\frac{|x|^2}{4}},\hat{T}\rangle\langle V,N\rangle\right)_i^j\\
    &=-\eta_i^j\left\langle\nabla e^{-\frac{|x|^2}{4}},N\right\rangle_i^j\left\langle V,\hat{T}\right\rangle_i^j+\eta_i^j\left\langle\nabla e^{-\frac{|x|^2}{4}},\hat{T}\right\rangle_i^j\left\langle V,N\right\rangle_i^j\\
    &=-\left\langle\nabla e^{-\frac{|x|^2}{4}},\mathcal{R}\hat{T}\right\rangle_i^j\left\langle V,\hat{T}\right\rangle_i^j+\left\langle\nabla e^{-\frac{|x|^2}{4}},\hat{T}\right\rangle_i^j\left\langle V,\mathcal{R}\hat{T}\right\rangle_i^j=-\left\langle \nabla e^{-\frac{|x|^2}{4}},\mathcal{R}V\right\rangle,\\
\end{split}
\end{equation}
where $\mathcal{R}$ is rotating by $\frac{\pi}{2}$ in the counter-clockwise direction. At any multi-junction $O_i$, this quantity is independent of $j$, as desired.
\end{proof}
\begin{rmk}
This computation is the same as in \cite{CM} except we need to check the boundary condition at the triple-junctions. Also, we have another eigenfunction $L\langle x,T\rangle=0$. This eigenfunction corresponds to a rotation in $\mathbb{R}^2$. 
\end{rmk}

\section{The variation direction of 4-ray star, 5-ray star}\label{sec:stars}
From the bilinear form
\begin{equation}
\begin{split}
    \sqrt{4\pi}[f,f]&=\int_\Gamma f\left(-Lf\right)e^{-\frac{|x|^2}{4}}d\sigma-\sum_{i,j}f_i^j\left\langle \nabla\left( fe^{-\frac{|x|^2}{4}}\right)_i^j,\hat{T}_i^j\right\rangle.
\end{split}
\end{equation}
We can think of the quantity as contributed by each curve $\gamma$ separately.
\begin{lem}\label{lem:const}
For any smooth AL-curve $\gamma$, and any function $f\in\mathfrak{V}$, the contribution of $\gamma$ to the bilinear form $\sqrt{4\pi}[f,f]$ is
\begin{equation}
\begin{split}
    \int_\gamma \left(-1+\frac{1}{4}\langle x,T\rangle^2\right)f^2e^{-\frac{|x|^2}{4}}+\left\langle \nabla f,\nabla \left(fe^{-\frac{|x|^2}{4}}\right)\right\rangle+\left\langle\nabla f,f \nabla e^{-\frac{|x|^2}{4}}\right\rangle d\sigma.\\
\end{split}
\end{equation}
In the special case that $f$ is a constant on $\gamma$, the contribution of $\gamma$ is
\begin{equation}
    \int_\gamma \left(-1+\frac{1}{4}\langle x,T\rangle^2\right)f^2 e^{-\frac{|x|^2}{4}}d\sigma.
\end{equation}
\end{lem}
\begin{proof}
For any smooth AL-curve $\gamma$, the boundary term is
\begin{equation}
\begin{split}
    -\sum_{\partial\gamma}&f\left\langle \nabla fe^{-\frac{|x|^2}{4}},\hat{T}\right\rangle=\int_\gamma \mathbf{div}\left(f\nabla \left(fe^{-\frac{|x|^2}{4}}\right)\right)d\sigma\\
    &=\int_\gamma f\mathbf{div}\left(e^{-\frac{|x|^2}{4}}\nabla f+f\nabla e^{-\frac{|x|^2}{4}}\right)+\left\langle\nabla f,\nabla \left(fe^{-\frac{|x|^2}{4}}\right)\right\rangle d\sigma\\
    &=\int_\gamma f\mathcal{L}fe^{-\frac{|x|^2}{4}}+f\left(\left\langle\nabla f, \nabla e^{-\frac{|x|^2}{4}}\right\rangle+f\Delta e^{-\frac{|x|^2}{4}} \right)+\left\langle\nabla f,\nabla \left(fe^{-\frac{|x|^2}{4}}\right)\right\rangle d\sigma.\\
    \end{split}
\end{equation}
The Laplacian of $e^{-\frac{|x|^2}{4}}$ on the curve $\gamma$ is given by
\begin{equation}
\begin{split}
    \Delta e^{-\frac{|x|^2}{4}}&=\mathbf{div}\left(-\frac{\langle x,T\rangle}{2} e^{-\frac{|x|^2}{4}}T\right)=\left(\frac{\langle x,T\rangle^2}{4}-\frac{\langle T,T\rangle}{2}-\frac{\langle x,kN\rangle}{2}\right)e^{-\frac{|x|^2}{4}}\\
    &=\left(\frac{1}{4}\langle x,T\rangle^2-\frac{1}{2}+k^2\right)e^{-\frac{|x|^2}{4}}.\\
\end{split}
\end{equation}
Combining both equation and use $Lf=\mathcal{L}f+(k^2+\frac{1}{2})f$ yields the desired result.
\end{proof}

Note that on a ray $\gamma$ starting at $r=a$ and goes to infinity, we have
\begin{equation}
    \int_\gamma \left(-1+\frac{1}{4}r^2 \right)e^{-\frac{r^2}{4}}dr=-\frac{1}{2}\left(\int_a^\infty e^{-\frac{r^2}{4}}dr-ae^{-\frac{r^2}{4}}\right)
\end{equation}
From the calculator, when $a<1.063$, this is negative. We have the following theorem:
\begin{thm}
If $\Gamma$ is the 4-ray star or the 5-ray star, we can find $\tilde{f}\in\mathfrak{V}$ which satisfies the following conditions:
\begin{enumerate}
    \item $\tilde{f}$ is orthogonal to $k$, $\langle V,N\rangle$, $V\in\mathbb{R}^2$ in $e^{-\frac{|x|^2}{4}}$-weighted $L^2$ norm.
    \item $[\tilde{f},\tilde{f}]<0$.
    \item $\tilde{f}$ is constant on each ray which goes to infinity.
\end{enumerate}
\end{thm}
\begin{figure}[H]
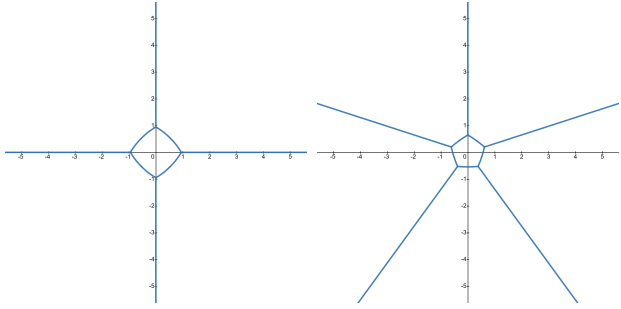

    \centering
    \includegraphics[width=4cm]{4cell.png}
    \includegraphics[width=4cm]{5cell.png}
    \caption{The 4-ray star and the 5-ray star. They are unstable shrinkers}
\end{figure}

\begin{proof}
Consider the space of functions $\mathcal{V}\subset\mathfrak{V}$ which contains all functions which are constant on each curve. The dimension of $\mathcal{V}$ is 4 for the 4-ray star and 5 for the 5-ray star.

From lemma \ref{lem:const}, the contribution of any curve $\gamma$ is given by $\int_\gamma \left(-1+\frac{1}{4}\langle x,T\rangle^2\right)f^2 e^{-\frac{|x|^2}{4}}d\sigma$. There are two cases, $\gamma$ is a ray or not. 

If $\gamma$ is a ray
\begin{equation}
\begin{split}
    \int_\gamma \left(-1+\frac{1}{4}\langle x,T\rangle^2\right)f^2 e^{-\frac{|x|^2}{4}}d\sigma=-\frac{1}{2}f^2\left(\int_a^\infty e^{-\frac{r^2}{4}}dr-ae^{-\frac{a^2}{4}}\right),\\
\end{split} 
\end{equation}
where $a$ is the starting $r$-value of the ray, from \cite{CG}, we have $a<1$ and this term is negative unless $f=0$.

If $\gamma$ is not a ray, from \cite{CG}, we know this curve is completely contained in the unit circle. Therefore, $\langle x,T\rangle\leq|x||T|<1$.
\begin{equation}
\begin{split}
    \int_\gamma \left(-1+\frac{1}{4}\langle x,T\rangle^2\right)f^2 e^{-\frac{|x|^2}{4}}d\sigma\leq\int_\gamma \left(-1+\frac{1}{4}\right)f^2 e^{-\frac{|x|^2}{4}}d\sigma\leq0\\
\end{split}
\end{equation}
This equals zero only when $f=0$.

Therefore, the quotient on $\mathcal{V}$ is strictly negative except for $f=0$. Since the space spanned by $k$ and $\langle V,N\rangle$, $v\in\mathbb{R}^2$ has dimension 3, there exist $\tilde{f}\in\mathcal{V}$ satisfied the desired conditions.
\end{proof}

\section{The variation direction of fish, rocket}\label{sec:fish_rocket}
We want to establish the following theorem.
\begin{thm}
If $\Gamma$ is the fish or the rocket, we can find $\tilde{f}\in\mathfrak{V}$ which satisfies the following conditions:
\begin{enumerate}
    \item $\tilde{f}$ is orthogonal to $k$, $\langle V,N\rangle$ in $e^{-\frac{|x|^2}{4}}$-weighted $L^2$ norm for all $V\in\mathbb{R}^2$.
    \item $[\tilde{f},\tilde{f}]<0$.
    \item $\tilde{f}$ is constant on each ray which goes to infinity.
\end{enumerate}
\end{thm}
The fish and the rocket are shown in the following picture.
\begin{figure}[H]
    \centering
    \includegraphics[width=4cm]{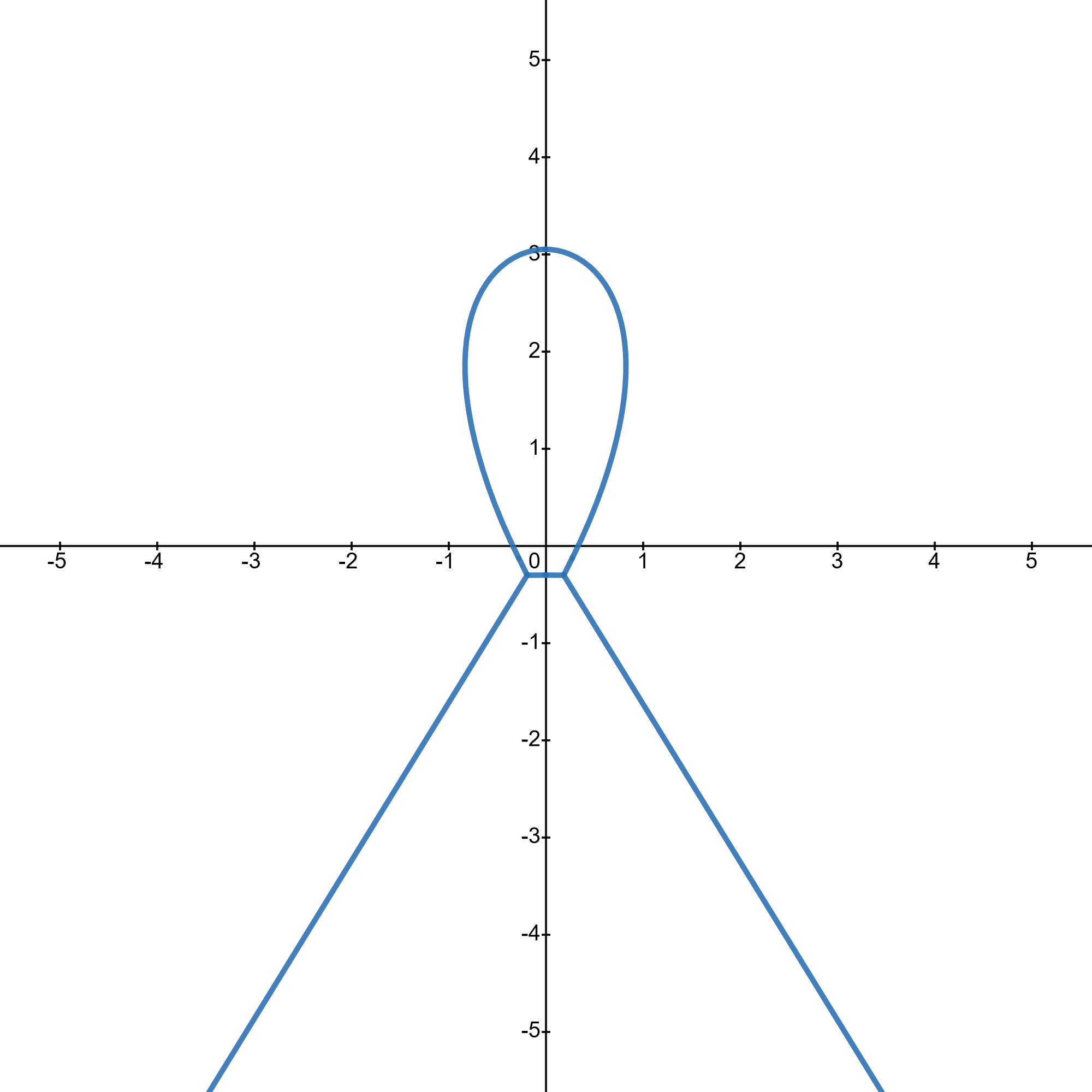}
    \includegraphics[width=4cm]{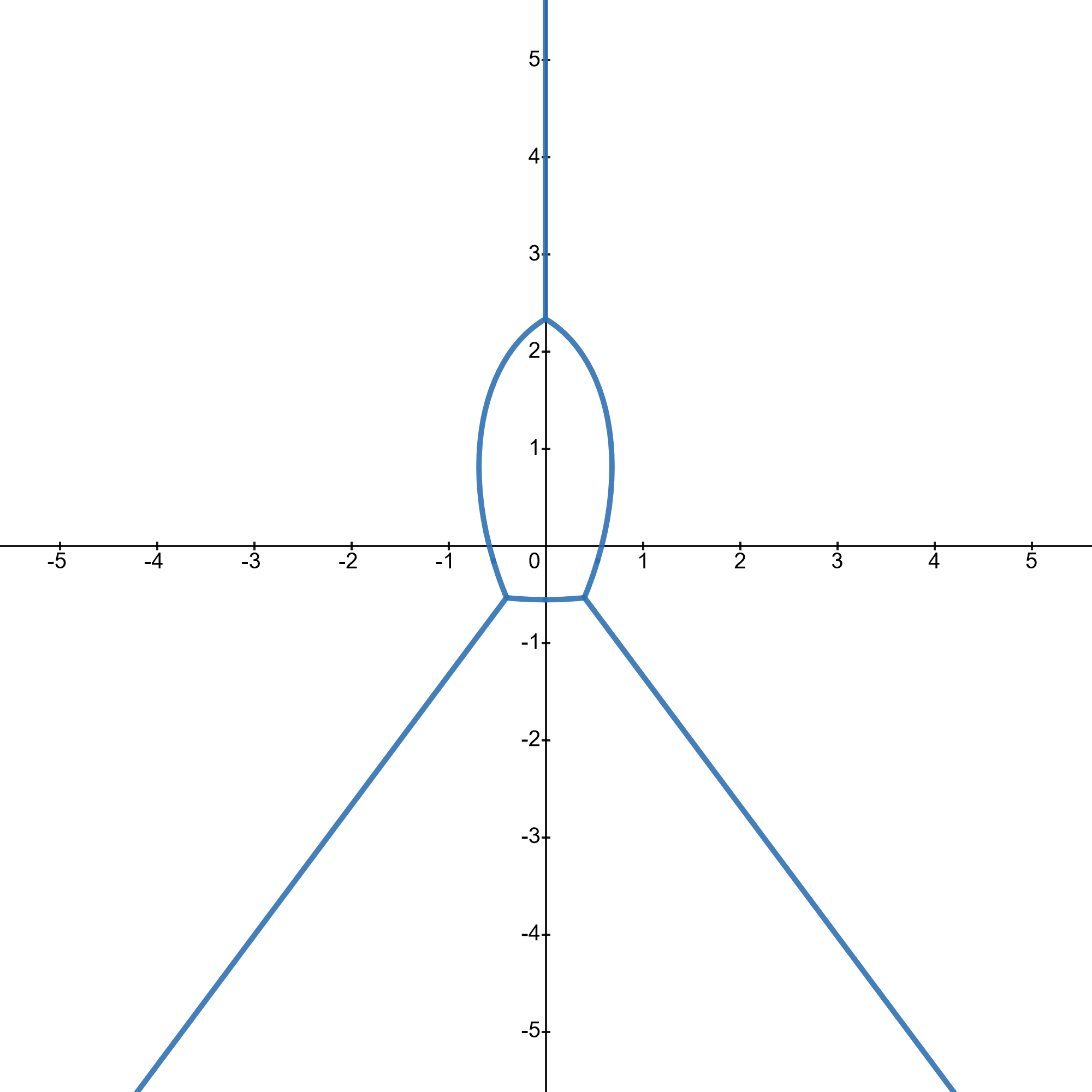}
    \caption{The fish and the rocket}
\end{figure}
\begin{proof}
We can rotate the fish and the rocket such that they are symmetric with respect to the $y$-axis with two rays pointing downward as in figure 4. Let the curve on top be $\gamma_1$ with the normal pointing inward. (Note that for the rocket, this is a union of two AL-curves.) Let the AL-curve on the bottom be $\gamma_2$ with the normal pointing outward. Let the ray on the bottom right be $\gamma_{3,R}$, with the normal pointing right and the ray on the bottom left be $\gamma_{3,L}$, with the normal pointing left. Let the triple junction on the right be $O_\mathbf{R}$ and the triple junction on the left be $O_\mathbf{L}$. For $j=1,2,3$, we have $\eta_{\mathbf{R}}^j=1$ and $\eta_{\mathbf{L}}^j=-1$

Let $f_1=k=-\frac{\langle x,N\rangle}{2}$, $f_2=\langle N,e_2\rangle$. We consider the function space $\mathcal{V}\subset\mathfrak{V}$ such that all function $v\in\mathcal{V}$ of the following form.
\begin{equation}
v=\displaystyle \begin{cases} 
a_1f_2+a_2f_2, & \text{on $\gamma_1$}, \\  
a_3, & \text{on $\gamma_2$}, \\  
a_4, & \text{on $\gamma_{3,R}$ and $\gamma_{3,L}$.}  
 \end{cases}
\end{equation}   
There are 4 coefficients with 1 constraint. $\mathcal{V}$ is a linear function space with dimension 3. For all $v\in \mathcal{V}$, 
\begin{equation}
    [v,v]=(\mathbf{I})+(\mathbf{II})+(\mathbf{III}),
\end{equation}
where
\begin{equation}
\begin{split}
    (\mathbf{I})=&\int_{\gamma_1}-\left(a_1f_1+a_2f_2\right)L\left(a_1f_1+a_2f_2\right)e^{-\frac{|x|^2}{4}}d\sigma\\
    &-\sum_{\partial\gamma_1}\left(a_1f_1+a_2f_2\right)\left\langle \nabla\left(\left(a_1f_1+a_2f_2\right)e^{-\frac{|x|^2}{4}}\right),\hat{T}\right\rangle,\\
    (\mathbf{II})=&\int_{\gamma_2}-a_3f_1L\left(a_3f_1\right)e^{-\frac{|x|^2}{4}}d\sigma-\sum_{\partial\gamma_2}a_3f_1\left\langle \nabla\left(a_3f_1e^{-\frac{|x|^2}{4}}\right),\hat{T}\right\rangle,\\
    (\mathbf{III})=&2\int_{\gamma_{3,R}}-a_4f_2L\left(a_4f_2\right)e^{-\frac{|x|^2}{4}}d\sigma-2\sum_{\partial\gamma_{3,R}}a_4f_2\left\langle \nabla\left(a_4f_2e^{-\frac{|x|^2}{4}}\right),\hat{T}\right\rangle.\\
\end{split}
\end{equation}

Since any function $f\in\mathcal{V}$ is constant on $\gamma_2$, $\gamma_{3,R}$, and $\gamma_{3,L}$ and the endpoints of these curves are contained in the unit circle, using lemma \ref{lem:const}, from the same argument as in the previous section, $(\mathbf{II})$ and $(\mathbf{III})$ are nonpositive. They are zero only when $a_3$, $a_4$ vanish respectively.

Now, we only need to show that $(\mathbf{I})$ is nonpositive. Using 
\begin{equation}\label{eq:by_parts}
\begin{split}
    &\int_{\gamma_1}-f_1Lf_2e^{-\frac{|x|^2}{4}}d\sigma-\sum_{\partial\gamma_1}f_1\left\langle \nabla\left(f_2e^{-\frac{|x|^2}{4}}\right),\hat{T}\right\rangle\\
    &=\int_{\gamma_1}-f_2Lf_1e^{-\frac{|x|^2}{4}}d\sigma-\sum_{\partial\gamma_1}f_2\left\langle \nabla\left(f_1e^{-\frac{|x|^2}{4}}\right),\hat{T}\right\rangle
\end{split}
\end{equation}
and $f_1=k$, $f_1e^{-\frac{|x|^2}{4}}$ is a constant on $\gamma_1$. We have
\begin{equation}
\begin{split}
    (\mathbf{I})&    =\left(-\int_{\gamma_1}f_1^2e^{-\frac{|x|^2}{4}}d\sigma\right)a_1^2+\left(-\int_{\gamma_1}f_1f_2e^{-\frac{|x|^2}{4}}d\sigma\right)2a_1a_2\\
    &\left(-\frac{1}{2}\int_{\gamma_1}f_2^2e^{-\frac{|x|^2}{4}}d\sigma-\sum_{\partial\gamma_1}f_2\left\langle \nabla\left(f_2e^{-\frac{|x|^2}{4}}\right),\hat{T}\right\rangle\right)a_2^2
    =Aa_1^2+2Ba_1a_2+Ca_2^2.
\end{split}
\end{equation}
This is a quadratic polynomial of $a_1$ and $a_2$. We want to show this term is negative unless $a_1=a_2=0$.

The coefficient of $a_1^2$ is 
\begin{equation}
    A=\int_{\gamma_1} -f_1^2e^{-\frac{|x|^2}{4}}d\sigma=-\int_{\gamma_1} k \frac{1}{2c}d\sigma=-\frac{\Delta\phi}{2c},
\end{equation}
where $c$ is the energy of the curve and $\Delta\phi$ is the change of angle of the tangent vector field on $\gamma_1$. Note that $h_1$ is the angle formed by the extension of $\gamma_{3,R}$ and $\gamma_{3,L}$ at the origin. For the fish, the change of angle is $\Delta\phi=\frac{5\pi}{3}-h_1$ and $-\frac{\Delta\phi}{2c}=-0.6149$. For the rocket, we need to compute the change of angle of two curves and add them together, $\Delta\phi=\frac{4\pi}{3}-h_1$ and $-\frac{\Delta\phi}{2c}=-0.7542$.

For the coefficient $B$ of the cross term $2a_1a_2$, since $f_1$, $f_2$ are eigenfunctions corresponding to different eigenvalues, they are orthogonal on $\Gamma$. We can expect the integration on $\gamma_1$ should be small. Use equation \eqref{eq:by_parts} together with $Lf_1=f_1$, $f_1e^{-\frac{|x|^2}{4}}=\frac{1}{2c}$ is a constant, $Lf_2=\frac{1}{2}f_2$, we have
\begin{equation}
    \int_{\gamma_1}-\frac{1}{2}f_1f_2e^{-\frac{|x|^2}{4}}d\sigma-\sum_{\partial\gamma_1}f_1\left\langle \nabla\left(f_2e^{-\frac{|x|^2}{4}}\right),\hat{T}\right\rangle=\int_{\gamma_1}-f_2f_1e^{-\frac{|x|^2}{4}}d\sigma.
\end{equation}
We can deduce
\begin{equation}
\begin{split}
  B&=-\int_{\gamma_1}f_1f_2e^{-\frac{|x|^2}{4}}d\sigma=-2\sum_{\partial\gamma_1}\left\langle f_1\nabla\left(e^{-\frac{|x|^2}{4}} f_2\right),\hat{T}\right\rangle=2\sum_{i=\mathbf{R},\mathbf{L}}\eta_i^1 f_1\left\langle \nabla e^{-\frac{|x|^2}{4}},-e_1\right\rangle\\
   &=2\cdot 2 f_1(O_{\mathbf{R}})e^{-\frac{|O_{\mathbf{R}}|^2}{4}}\left\langle\frac{O_\mathbf{R}}{2},e_1\right\rangle=\frac{1}{c}r_{\mathbf{in}}\sin\left(\frac{h_1}{2}\right),
\end{split}
\end{equation}
where the third equality is from equation \eqref{eq:f_2} with $V=e_2$ and $r_{\mathbf{in}}$ is the $r$-value for $O_\mathbf{R}$ and $O_\mathbf{L}$. Note that the same equation holds for the rocket since the boundary term for the triple-junction on top vanishes.
This value is $0.0554$ for the fish, $0.2050$ for the rocket.

The coefficient $C$ of $a_2^2$ is given by
\begin{equation}
    \int_{\gamma_1}-\frac{1}{2}f_2^2e^{-\frac{|x|^2}{4}}d\sigma-\sum_{\partial\gamma_1}f_2\left\langle \nabla\left(f_2e^{-\frac{|x|^2}{4}}\right),\hat{T} \right\rangle.
\end{equation}
We consider the boundary term first. For both fish and the rocket, at $O_{\mathbf{R}}$, we have 
\begin{equation}
\begin{split}
    f_2&=\langle N,e_2\rangle=\cos\left(\frac{h_1}{2}+\frac{\pi}{6}\right),\\
    \left\langle \nabla\left(f_2e^{-\frac{|x|^2}{4}}\right),\hat{T} \right\rangle&=\eta_\mathbf{R}^1\left\langle \nabla e^{-\frac{|x|^2}{4}},-e_1 \right\rangle=\frac{1}{2}r_{\mathbf{in}}e^{-\frac{r_{\mathbf{in}}^2}{4}}\sin\left(\frac{h_1}{2}\right).
\end{split}
\end{equation} 
We have the similar result for $O_{\mathbf{L}}$. Hence,
\begin{equation}
    C<-\sum_{\partial\gamma_1}f_2\left\langle \nabla\left(f_2e^{-\frac{|x|^2}{4}}\right),\hat{T} \right\rangle=-\cos\left(\frac{h_1}{2}+\frac{\pi}{6}\right)r_{\mathbf{in}}e^{-\frac{r_{\mathbf{in}}^2}{4}}\sin\left(\frac{h_1}{2}\right).
\end{equation}
This value is $-0.2415$ for fish and $-0.2124$ for the Rocket.

\begin{center}
    \begin{tabular}{|c|l|l|l|l|}
    \hline
    Name & $A$ & $B$ & $C$ & $B^2-AC$\\
    \hline  
    Fish & $-0.6149$ & $0.0554$ & $<-0.08562$ & $<-0.04957$\\
    \hline  
    Rocket & $-0.7542$ & $0.2050$ & $<-0.1417$ & $<-0.06484$\\
    \hline  
    \end{tabular}
\end{center}
From the value above, we can deduce that the term involving $a_1$, $a_2$ are negative unless $a_1=a_2=0$. Therefore, $[\cdot,\cdot]$ is negative definite on $\mathcal{V}$. We can find an even function $\tilde{f}$ in $\mathcal{V}$ which is orthogonal to $k$ and to $\langle N,e_2\rangle$. Since $\langle N, e_1\rangle$ is an odd function. The function in $\mathcal{V}$ is orthogonal to it automatically.
\end{proof}

\begin{rmk}
Even though the lens and the 3-ray star have the same topology as the fish and rocket, respectively, the estimation above cannot be applied to them since the triple junctions lie outside of the unit circle.
\end{rmk}

\section{$F$-unstableness of some regular shrinkers}\label{sec:F_unstable}
In previous two sections, we find a function $\tilde{f}$ which is orthogonal to $k$ and $\langle V,N\rangle$ for all $V\in\mathbb{R}^2$ in $e^{-\frac{|x|^2}{4}}$-weighted $L^2$ norm. Now, we are going to show this implies $\Gamma$ is $F$-unstable. We need to find a compactly supported function $\bar{f}$ such that for all $y$, $h$, $[(\bar{f},y,h),(\bar{f},y,h)]$ is negative. Recall
\begin{equation}
\begin{split}
    [(f,y,h),(f,y,h)]=&\left[\left[-fLf+f\langle y,N\rangle-2fhk-h^2k^2-\frac{\langle y,N\rangle^2}{2} \right]\right]\\
    &-\frac{1}{\sqrt{4\pi}}\sum_{i,j}f_i^j\left\langle\nabla\left(fe^{-\frac{|x|^2}{4}}\right)_i^j,\hat{T}_i^j\right\rangle.\\
\end{split}
\end{equation}

\begin{thm}
On a regular shrinker $\Gamma$, if there exist a function $\tilde{f}$ which satisfies
\begin{enumerate}
    \item $\tilde{f}$ is orthogonal to $k$, $\langle V,N\rangle$ in $e^{-\frac{|x|^2}{4}}$-weighted $L^2$ norm,
    \item $[\tilde{f},\tilde{f}]<0$,
    \item $\tilde{f}$ is constant on each ray which goes to infinity,
\end{enumerate}
then the regular shrinker $\Gamma$ is $F$-unstable. Therefore, the 4-ray star, 5-ray star, fish, and rocket are $F$-unstable regular shrinkers.
\end{thm}
\begin{proof}
We need to construct a compactly supported function $\bar{f}$ such that for all $y$, $h$, $[(\bar{f},y,h),(\bar{f},y,h)]$ is negative. Let $\phi(x)=\phi(|x|)$ is a smooth radial cutoff function such that $\phi(r)=1$ when $r<r_0$, $\phi(r)=0$ when $r>r_0$, $0\leq\phi(r)\leq1$, $|\phi'(r)|<2$, $|\phi''(r)|<4$ when $r_0\leq r\leq r_0+2$. We can choose $r_0$ large enough such that all triple junctions of $\Gamma$ are contained in $B_{r_0}$. Let $\bar{f}=\tilde{f}\phi=\tilde{f}-\hat{f}$, where $\hat{f}=\tilde{f}(1-\phi)$.
\begin{equation}
\begin{split}
    [(\bar{f}&,y,h),(\bar{f},y,h)]=[\bar{f},\bar{f}]+\left[\left[\bar{f}\langle y,N\rangle-2\bar{f}hk-h^2k^2-\frac{\langle y,N\rangle^2}{2} \right]\right]\\
    =&\left[\left[-\tilde{f}L\tilde{f}+\tilde{f}\langle y,N\rangle-2\tilde{f}hk-h^2k^2-\frac{\langle y,N\rangle^2}{2} \right]\right]\\
    &-\frac{1}{\sqrt{4\pi}}\sum_{i,j}\tilde{f}_i^j\left\langle\nabla\left(\tilde{f}e^{-\frac{|x|^2}{4}}\right)_i^j,\hat{T}_i^j\right\rangle+\left[\left[\hat{f}L\tilde{f}+\tilde{f}L\hat{f}-\hat{f}L\hat{f}-\hat{f}\langle y,N\rangle+2\hat{f}hk\right]\right].\\
\end{split}
\end{equation}
Note that from the property of $\tilde{f}$, $[[\tilde{f}k]]=[[\tilde{f}\langle y,N\rangle]]=0$, $[\tilde{f},\tilde{f}]<0$. Also $k$ is supported in $B_{r_0}$, $\hat{f}k=0$.
\begin{equation}
\begin{split}
    [(\bar{f},y,h)&,(\bar{f},y,h)]=[\tilde{f},\tilde{f}]+\left[\left[-h^2k^2-\frac{\langle y,N\rangle^2}{2} +\hat{f}L\tilde{f}+\tilde{f}L\hat{f}-\hat{f}L\hat{f}-\hat{f}\langle y,N\rangle\right]\right].\\
    =&[\tilde{f},\tilde{f}]-\left[\left[h^2k^2\right]\right]-\frac{1}{2}\left[\left[\left(\hat{f}-\langle y,N\rangle\right)^2\right]\right]+\left[\left[  \frac{1}{2}\hat{f}^2+\hat{f}L\tilde{f}+\tilde{f}L\hat{f}-\hat{f}L\hat{f}\right]\right].
\end{split}
\end{equation}
We need to estimate the last term. Note that since $\tilde{f}$ is constant on the rays, we have $L\tilde{f}=\frac{\tilde{f}}{2}$ outside $B_{r_0}$.
\begin{equation}
\begin{split}
    |L\hat{f}|&=\left|\tilde{f}\left(\Delta\phi-\left\langle\frac{x}{2},\nabla\phi\right\rangle+\frac{1}{2}\phi\right)\right|\leq\left|\tilde{f}\left(4+\frac{|x|}{2}\cdot 2+\frac{1}{2}\right)\right|\leq\left|\left(|x|+\frac{9}{2}\right)\tilde{f}\right|.
\end{split}    
\end{equation}
Therefore,
\begin{equation}
\begin{split}
    &\left|\sqrt{4\pi}\left[\left[  \frac{1}{2}\hat{f}^2+\hat{f}L\tilde{f}+\tilde{f}L\hat{f}-\hat{f}L\hat{f}\right]\right]\right|=\left|\int_\Gamma \left(\frac{1}{2}\hat{f}^2+\hat{f}L\tilde{f}+\tilde{f}L\hat{f}-\hat{f}L\hat{f}\right)e^{-\frac{|x|^2}{4}}d\sigma\right|\\
    &\leq \int_{\Gamma-B_{r_0}} \left(\frac{1}{2}\tilde{f}^2+\frac{1}{2}\tilde{f}^2+\left(|x|+\frac{9}{2}\right)\tilde{f}^2+\left(|x|+\frac{9}{2}\right)\tilde{f}^2\right)e^{-\frac{|x|^2}{4}}d\sigma\\
    &\leq \int_{\Gamma-B_{r_0}} \left(2|x|+10\right)\tilde{f}^2e^{-\frac{|x|^2}{4}}d\sigma.
\end{split}
\end{equation}
Since $\int_\Gamma (2|x|+10)\tilde{f}^2e^{-\frac{|x|^2}{4}}d\sigma$ is finite. We have $\int_{\Gamma-B_{r_0}} (2|x|+10)\tilde{f}^2e^{-\frac{|x|^2}{4}}d\sigma$ goes to 0 as $r_0$ goes to infinity. Choose $r_0$ large enough such that $\int_{\Gamma-B_{r_0}} (2|x|+10)\tilde{f}^2e^{-\frac{|x|^2}{4}}d\sigma<\frac{1}{2}|[\tilde{f},\tilde{f}]|$. The function $\bar{f}$ is compactly supported and for all $y$, $h$, we have
\begin{equation}
\begin{split}
    [(\bar{f},y,h)&,(\bar{f},y,h)]=-\left[\left[h^2k^2\right]\right]-\frac{1}{2}\left[\left[\left(\hat{f}-\langle y,N\rangle\right)^2\right]\right]\\
    &+[\tilde{f},\tilde{f}]+\left[\left[  \frac{1}{2}\hat{f}^2+\hat{f}L\tilde{f}+\tilde{f}L\hat{f}-\hat{f}L\hat{f}\right]\right]<0.
\end{split}
\end{equation}
Therefore, $\Gamma$ is $F$-unstable.
\end{proof}

Now, we use the following theorem in Colding and Minicozzi's work \cite{CM}.
\begin{thm*}{(\cite{CM})}
Suppose that $\Sigma\subset\mathbb{R}^{n+1}$ is a smooth complete self-shrinker with $\partial\Sigma=\emptyset$, with polynomial growth, and $\Sigma$ does not split off a line isometrically. If $\Sigma$ is $F$-unstable, then there is a compactly supported variation $\Sigma_s$ with $\Sigma_0=\Sigma$ so that $\lambda(\Sigma_s)<\lambda(\Sigma)$ for all $s\neq 0$.
\end{thm*}
We can obtain the following corollary.
\begin{cor}
The 4-ray star, 5-ray star, fish, and rocket are entropy unstable regular shrinkers.
\end{cor}

\section{Appendix: Some important properties for regular shrinkers}\label{sec:appendix}
In this section, we collect some properties of AL-curves which are used in section \ref{sec:stars} and \ref{sec:fish_rocket}. The reader may look up \cite{CG} and \cite{CL} for more details.

A curve satisfying $k=-\frac{\langle x,n\rangle}{2}$ is called an AL-curve since it is studied by Abresch and Langer in \cite{AL} to obtain properties of self-similarly shrinking solutions. Without loss of generality, we parametrize an AL-curve in the counterclockwise direction. Let $r,\theta$ be the parameter for polar coordinate in $\mathbb{R}^2$, $\phi$ be the direction of the tangent vector $T$ and $\psi$ be the signed angle from $\gamma$ to $T$. We have $\frac{dr}{ds}=\cos\psi$, $\frac{d\theta}{ds}=\frac{\sin\psi}{r}$, and $\frac{d\phi}{ds}=k=\frac{R\sin\psi}{2}$. Therefore, 
\begin{equation}
    \frac{d\psi}{ds}=\frac{d\phi}{ds}-\frac{d\theta}{ds}=\sin\psi(\frac{r}{2}-\frac{1}{r}).
\end{equation}
We can deduce
\begin{equation}
    \cot\psi d\psi=(\frac{r}{2}-\frac{1}{r})dr.
\end{equation}
After integration, we have $c\sin\psi=K(r)$, where $K(r)=\frac{e^{\frac{r^2}{4}}}{r}$. Here, $c$ is called the energy of the AL-curve. The curvature $k$ satisfies
\begin{equation}
    k=-\frac{\langle x,n\rangle}{2}=\frac{1}{2}r\sin\psi=\frac{1}{2c}\exp\left(\frac{r^2}{4}\right).
\end{equation}
On an AL-curve, $ke^{-\frac{|x|^2}{4}}=\frac{1}{2c}$ is a constant, as desired.

\begin{rmk}
The regular shrinkers in this work satisfies $k=-\frac{\langle x,N\rangle}{2}$. In \cite{CG} and \cite{CL}, the regular shrinker satisfies $k=-\langle x,N\rangle$. Therefore, the regular shinker is $\sqrt{2}$ times larger and the energy $c$ is $\sqrt{\frac{e}{2}}$ times larger than the corresponding quantities in \cite{CG} and \cite{CL}.
\end{rmk}

An AL-curve can connect to a ray at a triple junction. In this case, $\psi=\frac{\pi}{3}$ or $\frac{2\pi}{3}$. Let $r_{\mathbf{in}}<r_{\mathbf{out}}$ be the two solutions of $K(r)=c\sin\phi=\frac{\sqrt{3}c}{2}$. This are the possible $r$-value where an AL-curve connects to a ray at a triple junction. Now, from the work of Chen and Guo, we have the following numerical results.
\begin{center}
    \begin{tabular}{|c|l|l|l|l|l|l|}
    \hline
    Name & $c$ & $r_{\min}$ & $r_{\mathbf{in}}$ & $r_{\mathbf{out}}$ & $r_{\max}$ &  $h_1$\\
    \hline  
    Brakke spoon & 1.4021 & 0.8568 & 1.1390 & 1.7086 & 2.0596  & 1.9082\\
    \hline  
    Lens & 1.3938 & 0.8649 & 1.1590 & 1.6858 & 2.0487  & 1.9497 \\
    \hline  
    Fish & 3.3597 & 0.3046 & 0.3546 & 2.9271 & 3.0511  & 1.1040 \\
    \hline  
    3-ray star & 1.3716 & 0.8878 & 1.2251 & 1.6121 & 2.0180  & $\frac{2\pi}{3}$ \\
    \hline  
    Rocket & 1.9338 & 0.5591 & 0.6674 & 2.3358 & 2.5155  & 1.2717 \\
    \hline  
    4-ray star & 1.5281 & 0.7544 & 0.9443 & 1.9443 & 2.2038  & $\frac{2\pi}{4}$ \\
    \hline
    5-ray star & 1.9804 & 0.5436 & 0.6474 & 2.3675 & 2.5429  & $\frac{2\pi}{5}$ \\
    \hline  
    \end{tabular}
\end{center}

\end{document}